\documentclass[reqno,10pt, centertags]{amsart}
\usepackage{amsmath,amsthm,amscd,amssymb}
\usepackage{upref}
\usepackage{latexsym}
\usepackage{lscape}
\usepackage{graphicx}
\usepackage{multirow}
\usepackage{mathrsfs}
\usepackage{mathtools}
\mathtoolsset{showonlyrefs}
\usepackage{upref}
\usepackage{latexsym}
\usepackage{lscape}
\usepackage{graphicx}
\usepackage{tikz}
\usepackage{floatrow}
\usepackage{color}
\usepackage{enumitem}
\usepackage{mathtools}
\usepackage{tikz-cd}
\usepackage{tikz}
\usepackage[all,cmtip]{xy}
\usetikzlibrary{matrix}

\makeatletter
\def\theequation{\@arabic\c@equation}

\newcommand{\bbN}{{\mathbb{N}}}
\newcommand{\bbR}{{\mathbb{R}}}
\newcommand{\R}{{\mathbb{R}}}

\newcommand{\bbQ}{{\mathbb{Q}}}
\newcommand{\bbZ}{{\mathbb{Z}}}
\newcommand{\Z}{{\mathbb{Z}}}

\newcommand{\bbH}{{\mathbb{H}}}

\newcommand{\bbC}{{\mathbb{C}}}

\newcommand{\cB}{{\mathcal B}}

\newcommand{\cD}{{\mathcal D}}
\newcommand{\cE}{{\mathcal E}}
\newcommand{\cF}{{\mathcal F}}

\newcommand{\cH}{{\mathcal H}}
\newcommand{\cI}{{\mathcal I}}
\newcommand{\cJ}{{\mathcal J}}

\newcommand{\cM}{{\mathcal M}}
\newcommand{\cN}{{\mathcal N}}
\newcommand{\cO}{{\mathcal O}}

\newcommand{\cR}{{\mathcal R}}
\newcommand{\cS}{{\mathcal S}}

\newcommand{\cU}{{\mathcal U}}
\newcommand{\cV}{{\mathcal V}}
\newcommand{\cW}{{\mathcal W}}

\newcommand{\cZ}{{\mathcal Z}}

\newcommand{\bfi}{{\bf i}}

\newcommand{\no}{\nonumber}
\newcommand{\lb}{\label}

\newcommand{\wti}{\widetilde  }

\newcommand{\hatt}{\widehat}

\numberwithin{equation}{section}

\newcommand{\dom}{\operatorname{dom}}

\newcommand{\supp}{\operatorname{supp}}

\renewcommand{\Re}{\operatorname{Re }}
\renewcommand\Im{\operatorname{Im}}

\newtheorem{theorem}{Theorem}[section]

\newtheorem{lemma}[theorem]{Lemma}

\newtheorem{proposition}[theorem]{Proposition}

\theoremstyle{definition}

\newtheorem{example}[theorem]{Example}

\newtheorem{remark}[theorem]{Remark}

\begin{document}
	\begin{abstract}
	 We introduce a dynamically defined class of unbounded, connected, equilateral metric graphs on which the Kirchhoff Laplacian has zero Lebesgue measure spectrum and a nontrivial singular continuous part. A new local Borg--Marchenko uniqueness result is obtained in order to utilize Kotani theory for aperiodic subshifts satisfying Boshernitzan's condition.
	\end{abstract}
	
	\allowdisplaybreaks
	
	\title[Kirchhoff Laplacian with Zero Measure Spectrum]{Zero Measure and Singular Continuous  Spectra for Quantum Graphs}

	\author[D.\ Damanik]{David Damanik}
	\address{ Department of Mathematics, Rice University, Houston, TX 77005, USA}
	\email{{damanik@rice.edu}}
	\thanks{D.D.\ was supported in part by NSF grant DMS--1700131. }
	
	\author[L.\ Fang]{Licheng Fang}
	\address{Ocean University of China, Qingdao 266100, Shandong, China and Rice University, Houston, TX 77005, USA}
	\email{{lf18@rice.edu}}
	\thanks{L.F.\ was supported by NSFC (No. 11571327) and the Joint PhD. Scholarship Program of Ocean University of China.}
	
	\author[S.\ Sukhtaiev]{Selim Sukhtaiev}
	\address{ Department of Mathematics, Rice University, Houston, TX 77005, USA}
	\email{{sukhtaiev@rice.edu}}
	\thanks {S.S.\ was supported in part by an AMS-Simons travel grant, 2017-2019}
	\maketitle
	{\scriptsize{\tableofcontents}}
	\section{Introduction}
	\subsection{Overview} The spectral theory of Schr\"odinger operators with irregular potentials has been of great interest in mathematical physics since the 1950's.  A large number of models have been treated rigorously in the setting of the one-dimensional Laplacian perturbed by an irregular potential. As far as multidimensional phenomena are concerned, only a small fraction of the expected results has been proved. In this paper we focus on structures of intermediate dimensionality -- continuum metric graphs.  Specifically, we study the Kirchhoff Laplacian on aperiodic infinite volume graphs and demonstrate that they exhibit nontrivial spectral behavior.
	
	The first example of a quantum graph with ``exotic"  spectrum is due to Simon, cf. \cite{S96} where an infinite combinatorial graph with singular continuous spectrum was constructed.  Typically, interesting spectral phenomena (e.g., Anderson localization) occur due to irregular lower order perturbations of a fixed second order operator, which is often the Laplacian.   In the setting of graphs, however, it is natural to consider another type of perturbation -- by geometry. For example, the graph in Figure \ref{Fig1} may be viewed as a geometric perturbation, by inserting  diamond tiles, of the half line.  The latter has purely absolutely continuous spectrum $[0,\infty)$, while the former may exhibit all kinds of spectra depending on how the tiles are inserted. This paper concerns aperiodic geometric perturbations (of more general graphs) leading to zero Lebesgue measure spectra, a scenario of interest in the modeling of quasicrystals, \cite{DL2, LSS}.
	
	In the past two decades aperiodic Schr\"odinger operators on continuum and combinatorial graphs have attracted a significant amount of attention. The main interest has been around the relation between the geometry of graphs and their spectra. Let us mention some relevant results. In \cite{Br07} Breuer constructed a sparse combinatorial tree graph with singular continuous spectrum. This result was extended to the continuum setting by Breuer and Frank in \cite{BF}, where it was also shown that the singular continuous spectrum is in fact generic. Next, regular trees provided some insight into the Anderson model. For instance, delocalization in the regime of low disorder on the Bethe lattice was established by Klein in \cite{K3}. A new spectral behavior for the tree Anderson Hamiltonian near spectral edges was discovered by Aizenman and Warzel, \cite{AW11, AW13}. Dynamical localization for radial trees with disordered branching numbers and edge lengths was established by Damanik, Fillman, and Sukhtaiev in \cite{DFS, DS}. Periodicity of radial trees in the presence of absolutely continuous spectrum was shown by Exner, Seifert, and Stollmann in \cite{ESS}. Grigorchuk, Lenz, and Nagnibeda recently proved zero measure spectrum for {\it discrete} Laplacians on certain aperiodic graphs, cf. \cite{GLN}. To the best of our knowledge, continuum Kirchhoff Laplacians on {\it nontrivial}\footnote{noncompact and connected}  graphs with zero measure spectrum and nonempty singular continuous part have not yet been discussed in the literature. Our goal is to address this issue by combining a recent work of Kostenko and Nicolussi \cite{KN} (see also \cite{BKe, BL}) and the theory of ergodic Schr\"odinger operators \cite{CL, D2, DF}.
			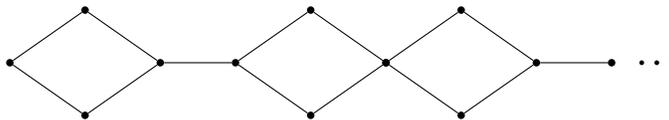
\begin{figure}[h]
		\centering
		\begin{tikzpicture}
		\node[fill=black,circle,scale=0.3](1) at (1,0){};
		\node[fill=black,circle,scale=0.3](2) at (2,0.7){};
		\node[fill=black,circle,scale=0.3](3) at (2,-0.7){};
		\node[fill=black,circle,scale=0.3](4) at (3,0){};
		\node[fill=black,circle,scale=0.3](5) at (4,0){};
		\node[fill=black,circle,scale=0.3](6) at (5,0.7){};
		\node[fill=black,circle,scale=0.3](7) at (5,-0.7){};
		\node[fill=black,circle,scale=0.3](8) at (6,0){};
		\node[fill=black,circle,scale=0.3](9) at (6,0){};
		\node[fill=black,circle,scale=0.3](10) at (7,0.7){};
		\node[fill=black,circle,scale=0.3](11) at (7,-0.7){};
		\node[fill=black,circle,scale=0.3](12) at (8,0){};
		\node[fill=black,circle,scale=0.3](13) at (9,0){};
		\draw[-,thin] (1) to (2);
		\draw[-,thin] (1) to (3);
		\draw[-,thin] (2) to (4);
		\draw[-,thin] (3) to (4);
		\draw[-,thin] (4) to (5);
		\draw[-,thin] (5) to (6);
		\draw[-,thin] (5) to (7);
		\draw[-,thin] (6) to (8);
		\draw[-,thin] (7) to (8);
		\draw[-,thin] (9) to (10);
		\draw[-,thin] (9) to (11);
		\draw[-,thin] (10) to (12);
		\draw[-,thin] (11) to (12);
		\draw[-,thin] (12) to (13);
		\filldraw (9.8,0) circle (.7pt);
		\filldraw (9.4,0) circle (.7pt);
		\filldraw (9.6,0) circle (.7pt);
		\end{tikzpicture}
		\caption{Fibonacci sphere numbers,\ $s=\{12112121...\}$}\lb{Fig1}
	\end{figure}

	\subsection{Setup} Let $\Gamma=(\cV,\cE)$ be a connected graph with root $o\in\cV$.  For $v\in\cV$, let $|v|$ denote the combinatorial length of the shortest path connecting $v$ and $o$. The combinatorial sphere of radius $n\geq0$ is given by $S_n:=\{v\in\cV: |v|=n\}$, the cardinality of this sphere is denoted by $s_n:=|S_n|$. In this paper we focus on graphs $\Gamma$ satisfying the following: Two vertices $u,v\in\cV$ are adjacent if and only if
	\[(u,v)\in (S_0\times S_1)\bigcup_{n\geq 1}S_n\times (S_{n-1}\cup S_{n+1}),\] see, e.g., Figure  \ref{Fig1} where $s_0=1$, $s_1=2$, $s_3=1$.
	We equip $\Gamma$ with a metric by assigning each edge $e\in\cE$  length $1$. A natural orientation is determined by the growth of spheres $S_n$. The Kirchhoff Laplacian on $\Gamma$ is defined by
	\begin{align}
	\begin{split}\lb{new1.1}
	&\bbH:\dom(\bbH)\subset L^2(\Gamma)\rightarrow L^2(\Gamma), \bbH u:= -u'',\ u\in\dom(\bbH),\\
	&\dom(\bbH)=\{f\in \hatt H^2(\Gamma)\cap C(\Gamma):
	\sum\limits_{e \in \mathcal E : v \in e}\partial_{\nu}^{e}f(v)=0, v \in \mathcal{V}\},
	\end{split}
	\end{align}
	where $\hatt H^2(\Gamma):=\bigoplus_{e\in\cE}H^2(e)$ is the $L^2$ based Sobolev space, cf. \cite{BK}.
	
	The main goal of this paper is to prove that the spectrum of $\bbH$ is a zero Lebesgue measure set whenever the sequence of sphere numbers satisfies certain dynamical conditions. Due to the large number of symmetries of $\Gamma$, the operator $\bbH$ can be written, see \cite[Theorem 3.5]{KN}, as the   direct sum of self-adjoint realizations of the following Sturm-Liouville differential expression
	\begin{align}
	\begin{split}
	&\tau = -\frac{1}{\mu(x)}\left(\frac{d}{dx}\mu(x)\frac{d}{dx}\right),\lb{n1.1n}\\
	&\mu(x):=\sum_{n=0}^{\infty}{s}_n {s}_{n+1}\chi_{[n, {n+1})}(x), x\in \bbR_+.
	\end{split}
	\end{align}
	The self-adjoint realizations $\tau$ in question are of two spectrally relevant types. The operators of the first type have compact resolvent and, hence, only discrete spectrum. They do enter the direct sum decomposition of $\bbH$ infinitely many times. However, there is only a finite collection of mutually non-unitarily equivalent operators of the first type.  Therefore, their total spectral contribution is worth only a zero Lebesgue measure set (as it is countable). Of course, some points of this set end up being eigenvalues of $\bbH$ of infinite multiplicity. Since they could be isolated from the rest of the spectrum of $\bbH$, we do not claim that $\sigma(\bbH)$ is a (generalized) Cantor set. The second type of self-adjoint realizations is given by $\tau$ acting in $L^2(0,\infty)$ subject to the Neumann condition at $0$.  Showing zero measure spectrum for this operator is the main technical issue addressed in this work.  To that end we utilize Kotani theory for $\tau$ and the base dynamical system $(\Omega, T)$ given by a subshift over the admissible values of the sphere numbers. Assuming that this subshift satisfies Boshernitzan's condition we obtain that the spectrum is given by the set $\cZ$ on which the Lyapunov exponent vanishes, see Section \ref{section4}.  This step, in particular, requires a Sch'nol-type result for $\tau$, see Lemma \ref{lemma2.2}, \cite{BMLS}. Then  we argue that $|\cZ|>0$ implies periodicity of the sequence of sphere numbers, contrary to the construction of $\Omega$. This gives that the spectrum of the half-line operator is a zero measure set, see Theorem~\ref{theorem4.5}, which in turn yields the main assertion.  A key to this argument is a new version of the celebrated Borg--Marchenko result, \cite{GL, GS, M, S}, establishing a one-to-one correspondence between the Weyl--Titchmarsh functions and the potentials. A rather general version of this result was obtained in \cite{EGNT1}; however, the operators of the form \eqref{n1.1n} are notably excluded from consideration, see \cite[Hypothesis 6.1]{EGNT1}. In this work, we prove a local version of Borg--Marchenko result that allows us to recover arbitrary finite blocks of sphere numbers from the asymptotic behavior of the $m$-functions for large non-real values of the spectral parameter, see Theorem~\ref{thm1.4}. It is worth mentioning that Breuer and Frank, \cite[Proposition 12]{BF}, obtained the Borg--Marchnko result for a somewhat relevant class of operators. Our method, which itself stems from Bennewitz's work \cite{Ben}, can easily be adapted to yield \cite[Proposition 12]{BF}.

	\section{Main Results}
	
Fix a finite set $\cJ\subset\bbN$. Let  $\cJ^{\bbZ}$ be equipped with the following metric
	 \begin{equation}\lb{3.21}
	 d(\omega, \wti \omega):=\sum_{n\in\bbZ}\frac{1-\delta_{\omega_n, \wti \omega_n}}{2^{|n|+1}},\ \omega, \wti \omega\in\cJ^{\bbZ}.
	 \end{equation}
	 Let $T: \cJ^{\bbZ}\rightarrow\cJ^{\bbZ}$ denote the left shift, i.e., $[T(\omega)]_n:=\omega_{n+1}$, $\omega\in\cJ^{\bbZ}$. A $T$-invariant, closed (with respect to $d$) subset $\Omega\subset \cJ^{\bbZ}$ is called a {\it subshift} over $\cJ$. This paper is concerned with a special class (which is yet very general\footnote{the scope of generality is comprehensively discussed in \cite{DL2}}) of subshifts satisfying Boshernitzan's condition $(B)$. Let us recall the relevant definitions. We say that $(\Omega, T)$ is {\it minimal} if every orbit $\{T^n\omega: \omega\in\Omega\}$ is dense. A minimal subshift is called {\it aperiodic} if one of its elements is not periodic, that is, $T^p \omega = \omega$ for some $p \in \Z$ implies $p = 0$. It is then easy to see that all elements of $\Omega$ are aperiodic. The set of words corresponding to $\Omega$ is defined by
	\[\cW:=\{\omega_k\cdots \omega_{k+n-1}:k\in\mathbb{Z},\ n\in\mathbb{N},\ \omega\in\Omega\}. \]
	Each word $w\in\cW$ of length $|w|\in\bbN$ determines the cylinder set
	\[V_{w}:=\{\omega\in\Omega: \omega_1\cdots \omega_{|w|}=w \}. \]
	For a $T$-invariant probability measure $\nu$ on $\Omega$, let us define the following quantity,
	\[\eta_{\nu}(n):=\min\{\nu(V_{w}): w\in \cW, |w|=n \}. \]
	A minimal subshift $(\Omega, T)$ is said to satisfy Boshernitzan's condition (B) if there is an ergodic probability measure $\nu$ on $\Omega$ such that
	\begin{equation}\lb{BC}
	\limsup\limits_{n\rightarrow\infty} n\,\eta_{\nu}(n)>0.
	\end{equation}

In the theorem below we denote $\Omega\upharpoonright_{\bbZ_+}:=\{ \{s_n\}_{n=0}^{\infty}: \{s_n\}_{n\in\bbZ}\in\Omega\}$.

	\begin{theorem}\lb{main}
		Suppose that $(\Omega, T)$ is a minimal aperiodic subshift over a finite set $\cJ\subset \bbN$, $|\cJ|\geq 2$ satisfying Boshernitzan's condition \eqref{BC}. Then for every $\{s_n\}_{n=0}^{\infty}\in\Omega\upharpoonright_{\bbZ_+}$, the spectrum of the Kirchhoff Laplacian $\bbH$ defined in \eqref{new1.1} is a zero Lebesgue measure set.
	\end{theorem}

	\begin{proof} 
		The Kirchhoff Laplacian $\bbH$ is unitarily equivalent, cf.~\cite[Theorem 3.5]{KN}, to the following direct sum
		\begin{equation}\lb{dec}
		H^+\oplus\bigoplus_{n\geq 1} \oplus _{j=1}^{s_{n}-1}  H_n^1 \oplus\bigoplus_{n\geq 1}\oplus_{j=1}^{(s_{n}-1)(s_{n+1}-1)}H_n^2,
		\end{equation}	
		where $H^+, H^1_n, H^2_n$ are the self-adjoint realizations of \eqref{n1.1n} on $\bbR_+=(0,\infty)$, $J_n:=(n-1, n+1),$ and $I_n:=(n, n+1)$, correspondingly, defined next.
		The half-line operator is given by
		\begin{align}
		\begin{split}\lb{nn1.1n}
		&H^+:\dom(H^+)\subset L^2(\bbR_+; \mu)\rightarrow L^2(\bbR_+; \mu), H^+ u:= \tau u,\ u\in\dom(H^+),\\
		&\dom(H^+)=\{f\in L^2(\bbR_+, \mu): f, \mu f'\in AC(\bbR_+), \tau f\in L^2(\bbR_+; \mu), f'(0)=0\}.
		\end{split}
		\end{align}
		The operators $H_n^1, H_n^2$, defined on finite intervals, are given by
		\begin{align}\lb{9.1}
		&H^1_n:\dom(H^1_n)\subset L^2(J_n; \mu)\rightarrow L^2(J_n; \mu), H^1_n u:= \tau u,\ u\in\dom(H^1_n),\\
		&\dom(H^1_n)=\left\{f\in L^2(J_n, \mu)\Big| \begin{matrix}
		f, \mu f'\in AC(J_n),
		\tau f\in L^2(J_n; \mu),\\ f(n-1)=f(n+1)=0
		\end{matrix}\right\},
		\end{align}
		and
		\begin{align}\lb{9.2}
		&H_n^2:\dom(H_n^2)\subset L^2(I_n; \mu)\rightarrow L^2(I_n; \mu), H_n^2 u:= \tau u,\ u\in\dom(H_n^2),\\
		&\dom(H_n^2)=\left\{f\in L^2(I_n, \mu)\Big| \begin{matrix}
		f, \mu f'\in AC(I_n),
		\tau f\in L^2(I_n; \mu),\\ f(n)=f(n+1)=0
		\end{matrix}\right\}.
		\end{align}
Since $s_n$ takes finitely many values, the sequence of operators $\{H^1_n\}_{n\in\bbZ_+}$ contains only finitely many mutually non-unitarily equivalent operators, each of which has only discrete spectrum (in fact the eigenvalues are given by the roots of some trigonometric transcendental equation, see \cite{Ber}). The union $\Sigma_1$ of these spectra is a countable set. Next, every operator $H_n^2$ is unitarily equivalent to the Dirichlet Laplacian on $(0,1)$, its spectrum is discrete and denoted by $\Sigma_2$. Then employing \eqref{dec} we get
\begin{equation}
\sigma(\bbH) = \sigma(H^+) \cup \Sigma_1\cup\Sigma_2.
\end{equation}
Thus, to prove the assertion it is enough to show that $|\sigma(H^+)|=0$. To that end, we extend the sequence $\{s_n\}_{n\geq 0}$ to a two-sided sequence belonging to the subshift $\Omega$, and introduce the corresponding self-adjoint operator on $\bbR$  as follows
\begin{align}
\begin{split}
&H :\dom(H)\subset L^2(\bbR; \mu )\rightarrow L^2(\bbR; \mu ), H  u:= \tau  u,\ u\in\dom(H ),\\
&\dom(H )=\{f\in L^2(\bbR, \mu ): f, \mu  f'\in AC(\bbR), \tau f\in L^2(\bbR; \mu )\},
\end{split}
\end{align}
where the function $\mu$ is defined for the two-sided sequence as in \eqref{n1.1n}. Now, by \cite[Theorem 9.11]{T09}, we have $\sigma_{ess}(H^+)\subset \sigma_{ess}(H)$\footnote{in fact, ``$\subset$'' can be replaced by ``$=$" by standard arguments utilizing ergodicity/minimality of $(\Omega,T)$, cf., e.g., \cite{DF}, \cite{CL}}. By Theorem~\ref{theorem4.5}, the Lebesgue measure of $\sigma_{ess}(H)$ is equal to zero (in fact, the whole spectrum is a fixed zero measure generalized Cantor set for every choice of the sequence of the sphere numbers from $\Omega$). Thus, we have $|\sigma(H^+)|\leq|\sigma_{ess}(H^+)|+|\sigma_{disc}(H^+)|=0$. Hence, one has
\[|\sigma(\bbH)|\leq |\sigma(H^+)|+|\Sigma_1|+|\Sigma_2|=0,\] as asserted.
\end{proof}

As pointed out in \cite{DL1, DL2}, many subshifts of interest satisfy condition (B). In fact, one can construct examples of such subshifts rather explicitly. Let us provide two examples.

\begin{example}
Let $N\geq 2$ and let $\cJ=\{\gamma_k; 1\leq k\leq N\}\subset \bbN$ be a set containing at least two distinct elements. 
Fix $\alpha\in\bbR\setminus\bbQ$ and  a rational partition of the unit circle $0=a_0<a_1<...<a_N=1$, $a_k\in\bbQ$. Pick any $\omega\in[0,1)$ and define
\[s_n(\omega):= \sum_{k=1}^{N}\gamma_k \chi_{[a_{k-1},a_k)}(n\alpha+\omega), n\in\bbZ. \]
Then the Kirchhoff Laplacian $\bbH$ given by such a sequence of sphere numbers has zero measure spectrum. Indeed, in this case the subshift $\Omega:=\overline{\{s(\omega): \omega\in[0,1)\}}$ satisfies Boshernitzan's condition \eqref{BC}, see e.g. \cite[Theorem 10]{DL2}. 
\end{example}

\begin{example}\lb{ex2.3}
Assume that the sequence of the sphere numbers $s\in \{1,2\}^{\bbZ_+}$ is invariant under the action of the Fibonacci substitution $\cS(1)=12$, $\cS(2)=1$, e.g. $s=12112121...$, see Figure \ref{Fig1}.  Then the Kirchhoff Laplacian $\bbH$ given by such a sequence of sphere numbers has zero measure spectrum. Here, 
\[\Omega:=\big\{\omega\in\cJ^{\bbZ}:\text{\ every subword of $\omega$ is a subword of $s$}\big\}.\] It satisfies Boshernitzan's condition \eqref{BC} by \cite[Theorem 4]{DL2}.
\end{example}

The sequences in both examples generate minimal subshifts satisfying (B), see \cite{D1, D2, DL1, DL2} for details.

\begin{remark}
The assertion of Theorem~\ref{main} also holds in the discrete setting, i.e., when $\bbH$ is given by the combinatorial graph Laplacian.  Up to a countable set, the spectrum of the discrete Hamiltonian is given by the spectrum of a certain Jacobi matrix, see \cite{BKe}.  Of course, the Borg--Marchenko and Sch'nol-type results become trivial at this level.
\end{remark}

	\section{Elements of Weyl--Titchmarsh Theory}
	In this section we discuss Weyl--Titchmarsh theory for $\tau$, cf. \eqref{n1.1n}, in order to prepare the necessary ingredients for Kotani theory. Specifically we prove a local version of the Borg--Marchenko uniqueness result in Theorem~\ref{thm1.4}, cf. \cite{Ben, B,  EGNT1, GL, GS, M, S}, and utilize it to construct the restriction maps in Theorem~\ref{thm3.6}. In order to simplify the notation, cf. \eqref{n1.1n}, we set
	\begin{equation}\lb{1.1n}
	\mu(x):=\sum_{n=0}^{\infty}{s}_n \chi_{[n, {n+1})}(x), x\in \bbR_+,
	\end{equation}
	where $\{s_n\}_{n\in\bbZ_+}\subset \cJ$ and $\cJ\subset\bbN$ is a finite set of cardinality at least two. The differential expression $\tau$ is defined as in \eqref{n1.1n}. The Dirichlet and Neumann solutions of the differential equation
	\begin{equation}\lb{1.1}
	\tau u-zu=0, z\in\bbC,
	\end{equation}
	are denoted by $\varphi(z, \cdot), \theta(z, \cdot)$, correspondingly, i.e.,
	\begin{equation}
	\varphi(z,0)=\mu(0)\theta'(z,0)=0,\  \mu(0)\varphi'(z,0)=\theta(z,0)=1.
	\end{equation}
	\begin{remark}\lb{1.1nn}
The conditions $f, \mu f'\in AC_{loc}(\bbR_+)$ imply
\begin{equation}
s_{n-1}f'(n^-)=s_{n}f(n^+), n\in\bbN.
\end{equation}
	\end{remark}
	
	Throughout this paper $\sqrt{z}$ denotes the branch of the square root corresponding to $\Re\sqrt{\bfi}>0$ and $\arg(z)\in(-\pi, \pi]$. In particular, $\Re(\sqrt{-z})>0$ whenever $z\in\bbC\setminus [0, \infty)$, and $\Re(\sqrt{-z})=0$ if $z\in[0,\infty)$.

	\subsection{Borg--Marchenko Theorem For Sturm--Liouville Operators in Impedance Form} The following lemma concerns the asymptotic behavior of the Dirichlet and Neumann solutions of \eqref{1.1} for large non-real spectral parameter.

	\begin{lemma}\lb{lemma1.1} For every $x\in\bbR_+\setminus\bbN$ one has
		\begin{align}
		\begin{split}&\varphi(z,x)=\frac{c(\lfloor x\rfloor)e^{\sqrt{-z}x}}{2\sqrt{-z}}(1+o(1)),\lb{1.2}\\
		&\varphi'(z,x)=\frac{c(\lfloor x\rfloor)e^{\sqrt{-z}x}}{2}(1+o(1)),
		\end{split}
		\end{align}
		\begin{align}
		\begin{split}\lb{1.2nn}
		&\theta(z,x)=\frac{c(\lfloor x\rfloor)e^{\sqrt{-z}x}}{2}(1+o(1)),\\
		&\theta'(z,x)=\frac{c(\lfloor x\rfloor)\sqrt{-z}e^{\sqrt{-z}x}}{2}(1+o(1)),
		\end{split}
		\end{align}
		as $z\rightarrow\infty$, $z\in\bbC\setminus[0,\infty)$,
		where
		\begin{equation}\lb{1.2nnn}
		c(\lfloor x\rfloor):=\frac{1}{s_0}\prod_{k=1}^{\lfloor x\rfloor}\frac{s_k+s_{k-1}}{2s_k}
		\end{equation}	
		Furthermore, for each $j\in\bbN$ one has
		\begin{equation}\lb{1.3}
		\varphi(z,x^+)=\cO\left(\frac{1}{\sqrt{z}}\right),\, \varphi'(z,x^+)=\cO\left(1\right),
		\end{equation}
		uniformly for $x\in[0,j)$ as $z\rightarrow\infty$, $z\in[0,\infty)$.
	\end{lemma}

	\begin{proof}[Proof of \eqref{1.2},  \eqref{1.2nn}]
To simplify the notation we denote
		$\alpha_k:=\frac{s_{k-1}}{s_k}$,
		and declare that all asymptotic formulas in this proof are for $z\rightarrow\infty$, $z\in\bbC\setminus[0,\infty)$.  Without loss we may assume that $s_0=1$.
		
		We will prove \eqref{1.2} by induction in $\lfloor x\rfloor$. If $\lfloor x\rfloor=0$, then $\tau u-zu=0$ is given by the free equation  $-u''(x)-zu(x)=0$, $x\in(0,1)$ and
		\begin{align}
		&\varphi(z,x^+)=\frac{\sinh(\sqrt{-z}x)}{\sqrt{-z}}=\frac{e^{\sqrt{-z}x}}{2\sqrt{-z}}(1+o(1)),\\
		&\varphi'(z,x^+)=\cosh\sqrt{-z}x=\frac{e^{\sqrt{-z}x}}{2}(1+o(1)),
		\end{align}
		as asserted.
		
		Suppose that both equations in \eqref{1.2}  hold for $\lfloor x\rfloor=k$, $k\in \bbZ_+$. Our first objective is to derive asymptotic formulas for $\varphi(z, (k+1)^+)$ and $\varphi'(z, (k+1)^+)$. Since
		\begin{equation}\lb{n1.4}
		\left(\cosh\sqrt{-z}x; \frac{\sinh\sqrt{-z}x}{\sqrt{-z}}\right)
		\end{equation}
		is a fundamental system for the differential equation \eqref{1.1} considered on $(k, k+1)$, and, by Remark \ref{1.1nn},
		\begin{equation}\lb{nn1.4}
\varphi(z,(k+1)^+)=\varphi(z, (k+1)^-),\ \varphi'(z,(k+1)^+)=\alpha_{k+1}\varphi'(z, (k+1)^-),
		\end{equation}
		the vector $(\varphi, \varphi')^{\top}(z,(k+1)^+)$ may be recovered using the fundamental matrix corresponding to \eqref{n1.4} and the matrix of vertex conditions corresponding to \eqref{nn1.4} as follows
		\begin{align}
		&\begin{bmatrix}
		\varphi(z, (k+1)^+)\\
		\varphi'(z, (k+1)^+)
		\end{bmatrix}
		=
		\begin{bmatrix}
		\cosh\frac{\sqrt{-z}}{2}& \frac{\sinh\frac{\sqrt{-z}}{2}}{\sqrt{-z}}\\
		{\alpha_{k+1}} {\sqrt{-z}\sinh\frac{\sqrt{-z}}{2}}&{\alpha_{k+1}}\cosh\frac{\sqrt{-z}}{2}
		\end{bmatrix}
		\begin{bmatrix}
		\varphi(z, k+1/2)\\
		\varphi'(z, k+1/2)
		\end{bmatrix}.
		\end{align}
		Combining this and the induction hypothesis we get
		\begin{align}
		\begin{split}\lb{1.4}
		&\varphi(z, (k+1)^+)=\cosh\left(\frac{\sqrt{-z}}{2}\right)\varphi\left(z, k+\frac12\right)+ \frac{\sinh\frac{\sqrt{-z}}{2}}{\sqrt{-z}}\varphi'\left(z, k+\frac12\right)\\
		 &\quad=\left(\frac{e^{\frac{\sqrt{-z}}2}}{2}\frac{e^{\sqrt{-z}(k+1/2)}}{2\sqrt{-z}}+\frac{e^{\frac{\sqrt{-z}}2}}{2\sqrt{-z}}\frac{e^{\sqrt{-z}(k+1/2)}}{2}\right)c_k(1+o(1))\\
		&\quad= \frac{c_ke^{\sqrt{-z}(k+1)}}{2\sqrt{-z}}(1+o(1)).
		\end{split}
		\end{align}
		Likewise,
		\begin{align}
		\begin{split}\lb{1.5}
		&\varphi'(z, (k+1)^+)={\alpha_{k+1}}{\sqrt{-z}}\,\sinh\left(\frac{\sqrt{-z}}{2}\right)\varphi\left(z, k+\frac12\right)\\
		&\hspace{4cm}+{\alpha_{k+1}}\cosh\left(\frac{\sqrt{-z}}{2}\right)\varphi'\left(z, k+\frac12\right)\\
		&\quad={\alpha_{k+1}} c_k\left(\frac{\sqrt{-z}e^{\frac{\sqrt{-z}}2}}{2}\frac{e^{\sqrt{-z}(k+1/2)}}{2\sqrt{-z}}+\frac{e^{\frac{\sqrt{-z}}2}}{2}\frac{e^{\sqrt{-z}(k+1/2)}}{2}\right)(1+o(1))\\
		&\quad= \frac{{\alpha_{k+1}}c_ke^{\sqrt{-z}(k+1)}}{2}(1+o(1)).
		\end{split}
		\end{align}
		Using the fundamental matrix corresponding to \eqref{n1.4} as before we get
		\begin{align}
		&\begin{bmatrix}
		\varphi(z, x)\\
		\varphi'(z, x)
		\end{bmatrix}
		=
		\begin{bmatrix}
		\cosh{\sqrt{-z}(x-k)}& \frac{\sinh\sqrt{-z}(x-k)}{\sqrt{-z}}\\
		{\sqrt{-z}\sinh\sqrt{-z}(x-k)}& \cosh\sqrt{-z}(x-k)
		\end{bmatrix}
		\begin{bmatrix}
		\varphi(z, k^+)\\
		\varphi'(z, k^+)
		\end{bmatrix},
		\end{align}
		for every $x\in(k+1,k+2)$. Then combining this and the asymptotic formulas \eqref{1.4}, \eqref{1.5} we arrive at
		\begin{align}
		\begin{split}
		&\varphi(z, x)=\cosh(\sqrt{-z}(x-k))\varphi\left(z, k^+\right)+ \frac{\sinh{\sqrt{-z}(x-k)}}{\sqrt{-z}}\varphi'\left(z, k^+\right)\\
		 &\quad=\left(\frac{e^{\sqrt{-z}(x-k)}}{2}\frac{e^{\sqrt{-z}k}}{2\sqrt{-z}}+\frac{e^{\sqrt{-z}(x-k)}}{2\sqrt{-z}}\frac{{\alpha_{k+1}}e^{\sqrt{-z}k}}{2}\right)c_k(1+o(1))\\
		&\quad= \frac{c_ke^{\sqrt{-z}x}}{2\sqrt{-z}}\left(\frac{1+{\alpha_{k+1}}}{2}\right)(1+o(1))=\frac{c_{k+1}e^{\sqrt{-z}x}}{2\sqrt{-z}}(1+o(1)),
		\end{split}
		\end{align}
		and
		\begin{align}
		\begin{split}
		&\varphi'(z, x)= {\sqrt{-z}}\,\sinh(\sqrt{-z}(x-k))\varphi(z, k^+)+ \cosh(\sqrt{-z}(x-k))\varphi'\left(z, k^+\right)\\
		&\quad=  c_k\left(\frac{\sqrt{-z}e^{\sqrt{-z}(x-k)}}{2}\frac{e^{\sqrt{-z}k}}{2\sqrt{-z}}+\frac{e^{\sqrt{-z}(x-k)}}{2}\frac{{\alpha_{k+1}}e^{\sqrt{-z}k}}{2}\right)(1+o(1))\\
		&\quad= \frac{c_ke^{\sqrt{-z}x}}{2}\left(\frac{1+{\alpha_{k+1}}}{2}\right)(1+o(1))=\frac{c_{k+1}e^{\sqrt{-z}x}}{2}(1+o(1)).
		\end{split}
		\end{align}
		Therefore \eqref{1.2} holds as asserted. The asymptotic formulas in \eqref{1.2nn} can be shown similarly.
	\end{proof}

	\begin{proof}[Proof of \eqref{1.3}]  All asymptotic formulas in this proof are for  $z\rightarrow\infty$, $z\in[0,\infty)$. We will prove  \eqref{1.3} by induction in $j$.  If $j=1$, then
		\begin{align}
		&\varphi(z,x^+)=\frac{\sin(\sqrt{z}x)}{\sqrt{z}},\varphi'(z,x^+)=\cos\sqrt{z}x,
		\end{align}
		and thus \eqref{1.3} holds.
		
		Assuming that \eqref{1.3} holds for some $j\in\bbN$, we will prove that it also holds for $j+1$. Since $\varphi$ satisfies the vertex conditions, cf. \eqref{nn1.4}, and since $\cos(\sqrt{z}x), \frac{\sin(\sqrt{z}x)}{\sqrt{z}}$ is a fundamental system for the differential equation \eqref{1.1} considered on $(j-1,j)$, we get
		\begin{align}
		&\begin{bmatrix}
		\varphi(z, j^+)\\
		\varphi'(z, j^+)
		\end{bmatrix}
		=
		\begin{bmatrix}
		\cos\frac{\sqrt{z}}{2}& \frac{\sin\frac{\sqrt{z}}{2}}{\sqrt{z}}\\
		{-{\alpha_{j}}\sqrt{z}\sin\frac{\sqrt{z}}{2}}&{\alpha_{j}}\cos\frac{\sqrt{z}}{2}
		\end{bmatrix}
		\begin{bmatrix}
		\varphi(z, j-1/2)\\
		\varphi'(z, j-1/2)
		\end{bmatrix}.
		\end{align}
		Using this and the induction hypothesis we obtain
		\begin{align}
		\begin{split}
		&\varphi(z, j^+)=\cos\left(\frac{\sqrt{z}}{2}\right)\varphi\left(z, j-\frac12\right)+ \frac{\sin\frac{\sqrt{z}}{2}}{\sqrt{z}}\varphi'\left(z, j-\frac12\right)\\
		&\quad=\cos\frac{\sqrt{z}}{2}\cO\left(\frac{1}{\sqrt{z}}\right)+ \frac{\sin\frac{\sqrt{z}}{2}}{\sqrt{z}}\cO\left(1\right)=\cO\left(\frac{1}{\sqrt{z}}\right)
		\end{split}
		\end{align}
		and
		\begin{align}
		\begin{split}
		&\varphi'(z, j^+)=-\alpha_j{\sqrt{z}}\,{\sin\left(\frac{\sqrt{z}}{2}\right)}\varphi\left(z, j-\frac12\right)+\alpha_j\cos\left(\frac{\sqrt{z}}{2}\right)\varphi'\left(z, j-\frac12\right)\\
		 &\quad=-\alpha_j{\sqrt{z}}\,{\sin\left(\frac{\sqrt{z}}{2}\right)}\cO\left(\frac{1}{\sqrt{z}}\right)+\alpha_j\cos\left(\frac{\sqrt{z}}{2}\right)\cO\left(1\right)=\cO\left(1\right).
		\end{split}
		\end{align}
		Combining these asymptotic formulas and
		\begin{align}
		&\begin{bmatrix}
		\varphi(z, x)\\
		\varphi'(z, x)
		\end{bmatrix}
		=
		\begin{bmatrix}
		\cos\sqrt{z}(x-j)& \frac{\sin\sqrt{z}(x-j)}{\sqrt{z}}\\
		{- \sqrt{z}\sin\sqrt{z}(x-j)}& \cos\sqrt{z}(x-j)
		\end{bmatrix}
		\begin{bmatrix}
		\varphi(z, j^+)\\
		\varphi'(z, j^+)
		\end{bmatrix},
		\end{align}
		$x\in (j, j+1)$ one infers \eqref{1.3} with $j$ replaced by $j+1$ as required.
	\end{proof}
	
	\begin{remark}
		Choosing $s_k\equiv 1$ leads to the free Laplacian in which case \eqref{1.2}, \eqref{1.2nn} are consistent with the explicit solutions to the unperturbed problem.
	\end{remark}
	
	Let us recall the definition of the Weyl--Titchmarsh function $m$ corresponding to $\tau$.  First, for $z\in\bbC\setminus\bbR$ there exists a unique (up to a scalar multiple) square integrable solution $\psi(z,\cdot)\in L^2(\bbR_+; \mu)$ of \eqref{1.1}. Since $z$ is non-real, we have $\varphi(z,\cdot)\not\in L^2(\bbR_+; \mu)$ (otherwise $z$ would be an eigenvalue of the Dirichlet realization of $\tau$). Thus there exists a non-zero $m(z)\in\bbC$ such that
	\begin{equation}\lb{1.5n}
	\psi(z,x)= \theta(z,x)+m(z)\varphi(z,x).
	\end{equation}
	The Weyl--Titchmarsh function $m$ is the mapping $z\mapsto m(z)$. It is analytic in $\bbC_+:=\{z\in\bbC: \Im z>0\}$ and, in fact, it is a Herglotz function, i.e., $\Im m(z)>0$, $z\in\bbC_+$.

	\begin{lemma}
		For every $x\in\bbR_+$, $\alpha\in(-\pi,\pi)\setminus\{0\}$ one has
		\begin{equation}\lb{1.7}
		\varphi(z,x)\psi(z,x)\rightarrow 0,
		\end{equation}
		as $z\rightarrow\infty$ along the ray $\arg(z)=\alpha$.
	\end{lemma}

	\begin{proof}
		Our strategy is to view the left-hand side of \eqref{1.7} as the diagonal Green function. In order to show that the latter vanishes asymptotically for high energies we will employ \eqref{1.3} and the following representation of $m$, see e.g. \cite[Chapter 9]{T09},
		\begin{align}
		&m(z)=d+\int_{\bbR}\left(\frac{1}{\lambda-z}-\frac{\lambda}{1+\lambda^2}\right) d\nu(\lambda), \\
		&d= \Re(m(\bfi)),\ \int_{\bbR}\frac{1}{1+\lambda^2}d\nu(\lambda)=\Im(m(\bfi))<\infty.\lb{1.8}
		\end{align}
		The Green function of $H^+$, cf. \eqref{nn1.1n}, that is, the integral kernel of $(H^+-z)^{-1}$, admits two representations,
		\begin{equation}\lb{1.9}
		G(z,x,x)=\varphi(z,x)\psi(z,x), x\in\bbR_+,
		\end{equation}
		and
		\begin{equation}
		G(z,x,x)=\int_{\bbR}\frac{\varphi^2(\lambda, x)}{z-\lambda}d\nu(\lambda), x\in\bbR_+.
		\end{equation}
		Then utilizing \eqref{1.3}, the second equation in \eqref{1.8}, the inequality (cf. \cite[(3.114)]{T09})
		\begin{equation}
		\left|\frac{1}{\lambda-z}\right|\leq \frac{1}{1+|\lambda|}\frac{1+|z|}{\Im(z)},
		\end{equation}
		and the dominated convergence theorem, we get $G(z, x,x)\rightarrow0$ as $z\rightarrow\infty$, $\arg(z)=\alpha$. Hence by \eqref{1.9} we get \eqref{1.7}.
	\end{proof}

	The proof of the following theorem is motivated by \cite{Ben}, where the Borg--Marchenko Theorem is established for Schr\"odinger operators $-\frac{d^2}{dx^2}+V(x)$. The main ingredients of Bennewitz's proof are certain asymptotic formulas for the Dirichlet and Neumann solutions $\varphi(z,x), \theta(z,x)$ and for the Green function $G(z,x,x)$. In our setting these are given by \eqref{1.2}, \eqref{1.2nn} and \eqref{1.7}, respectively.

	\begin{theorem}\lb{thm1.4} Let $s, \wti s\in\cJ^{\bbZ_+}$ and let $m, \wti m$ be the corresponding Weyl--Titchmarsh functions. Then
		\begin{equation}\lb{1.9n}
		s_j=\wti s_j,\ 0\leq j\leq k,
		\end{equation}
		if and only if for every $\alpha\in (-\pi,\pi)\setminus\{0\}$, one has
		\begin{equation}\lb{1.9nn}
		|m(z)-\wti m(z)|=o\left(\exp\{-2\left(k+1/2\right)\Re\sqrt{-z}\}\right),
		\end{equation}
		as $z\rightarrow\infty$ along the ray $\arg(z)=\alpha$.
		\end{theorem}
	
	\begin{proof} Let $\varphi(z, \cdot)$, $\wti \varphi(z, \cdot)$, $\theta(z, \cdot)$, $\wti\theta(z, \cdot)$ be the Dirichlet and Neumann solutions corresponding to $s$ and $\wti s$ respectively.	 Then  for every $x\in\bbR_+\setminus\bbN$ the first equation in \eqref{1.2} yields
		\begin{equation}\lb{1.10}
		\frac{\varphi(z, x)}{\wti \varphi(z, x)}\rightarrow\frac{c(\lfloor x\rfloor)}{\wti c(\lfloor x\rfloor)},\text{\ as\ } z\rightarrow\infty,z\in\bbC\setminus[0,\infty),
		\end{equation} with $c,\wti c$ as in \eqref{1.2nnn}.
		Using the fact that $\frac{c(\lfloor x\rfloor)}{\wti c(\lfloor x\rfloor)}$ is a nonzero constant with respect to $z$ and \eqref{1.7}, \eqref{1.10} we obtain
		\begin{equation}
		\wti \varphi (z, x)\psi(z,x)\rightarrow0,\  \varphi (z, x)\wti \psi(z,x)\rightarrow0,
		\end{equation}
		as $z\rightarrow\infty$ along an arbitrary non-real ray. Therefore for every $x\in\bbR_+\setminus\bbN$, \eqref{1.5n} yields
		\begin{align}\lb{nn1.12}
		\wti \varphi (z, x)\theta(z,x)-\varphi (z, x)\wti \theta(z,x)
		+(m(z)-\wti m(z))\varphi (z, x)\wti \varphi (z, x)\rightarrow 0,
		\end{align}
		as $z\rightarrow\infty$ along an arbitrary non-real ray.

		We are ready to prove the``if" part.  Combining \eqref{1.2}, \eqref{1.9nn} and \eqref{nn1.12}, and fixing $x\in[0,k]\setminus\bbN$, we obtain
		\begin{equation} \lb{n1.12}
		\wti \varphi (z, x)\theta(z,x)-\varphi (z, x)\wti \theta(z,x)\rightarrow 0,
		\end{equation}
		as $z\rightarrow\infty$ along every non-real ray. Let us introduce
		\begin{equation}
		f(z):= \wti \varphi (z, x)\theta(z,x)-\varphi (z, x)\wti \theta(z,x).
		\end{equation}
		This is an entire function whose growth rate is at most $1/2$, cf.~\eqref{1.2}, \eqref{1.2nn}, and it is bounded along every non-real ray. Then, by the Phragm\'en--Lindel\"of principle, $f$ is bounded on $\bbC$. Thus it is constant and \eqref{n1.12} gives $f(z)=0$ for all $z\in\bbC$.
		
		That is, for all $x\in[0,k]\setminus\bbN$, $z\in\bbC$, one has
		\begin{equation}
		\frac{\wti \theta(z,x)}{\wti \varphi (z, x)}=\frac{ \theta(z,x)}{\varphi (z, x)}, x\in[0,k]\setminus\bbN.
		\end{equation}
		Taking the derivative with respect to $x$ and recalling that the Wronskian is conserved, we get
		\begin{equation}
		\frac{1}{\wti \mu (x^+)\wti \varphi ^2(z, x^+)}=\frac{1}{\mu(x^+)\varphi ^2(z, x^+)}.
		\end{equation}
		Hence, for each $0\leq j\leq k$
		\begin{equation}\lb{1.12}
		{\wti s_j \wti \varphi ^2(z, x^+)}={s_j \varphi ^2(z, x^+)},\
		\end{equation}
		and upon (right-)differentiating with respect to $x$
		\begin{equation}\lb{1.13}
		{\wti s_j \wti \varphi (z, x^+)}={s_j \varphi (z, x^+)}.
		\end{equation}
		Then $\eqref{1.12}$ and $\eqref{1.13}$ give
		\begin{equation}
		\left(\frac{\wti s_j}{s_j}\right)^2=\frac{\wti s_j}{s_j} \text{\ and\ }\wti s_{j}=s_j,\ 0\leq j\leq k,
		\end{equation}
		as asserted.
		
		Next, we prove the ``only if" part. Using \eqref{1.9n}  we get
		\begin{equation}
		\varphi(z,x)=\wti\varphi(z,x),\ \theta(z,x)=\wti\theta(z,x), x\in[0, k+1), z\in\bbC.
		\end{equation}
		Then \eqref{nn1.12} with $x\in [0, k+1/2]\setminus\bbN$ yields
		\begin{equation}
		|(m(z)-\wti m(z))\varphi (z, x)\wti \varphi (z, x)|=o(1),
		\end{equation}
		$z\rightarrow\infty$ along every non-real ray. Combining this and the first formula in \eqref{1.2} we get \eqref{1.9nn}.
	\end{proof}

	\subsection{Continuity of $m$-functions With Respect to Jumps.}

	\begin{lemma}\lb{lemma3.5}
		Fix a sequence $\{s(n)\}_{n=1}^{\infty}\subset\cJ^{\bbZ_+}$.  If $s(n)\rightarrow s$ as $n\rightarrow\infty$ with respect to the metric $d$, then $m_n(z)\rightarrow m(z)$ as $n\rightarrow\infty$ uniformly on compact subsets of $\bbC_+$.
	\end{lemma}

	\begin{proof}
		The proof is based on the fact that the Weyl circles (more precisely their centers and radii) are continuous with respect to $z$ and eventually constant with respect to $n$.
		
		 Let $m(z,b, \alpha)\in\bbC$ be such that the function
		\begin{equation}
		\chi(z,x):=\theta(z,x)+m(z,b,\alpha) \varphi (z,x),\ x\in(0,b)
		\end{equation}
		satisfies the following condition at $b>0$,
		\begin{equation}
		\chi(z,b)\cos\alpha+\mu (b)\chi(z,b) \sin\alpha =0.
		\end{equation}
		Then
		\begin{equation}
		C(z,b):=\{m(z,b,\alpha): \alpha\in[0,\pi] \},
		\end{equation}
		is a circle in $\bbC$ centered at
		\begin{equation}
		O(z,b):=-\frac{W(\varphi(z,\cdot), \overline{\theta(z,\cdot)})(b)}{W(\theta(z,\cdot), \overline{\theta(z,\cdot)})(b)},
		\end{equation}
		with radius
		\begin{equation}\lb{1.15}
		r(z,b):=\frac{1}{|W(\theta(z,\cdot),\overline{\theta(z,\cdot)})(b)|}\rightarrow0, b\rightarrow\infty.
		\end{equation}
		The corresponding disk is denoted by $D(z,b)$. Then $D(z,b')\subset D(z,b)$ whenever $b'>b>0$, and $m(z)\in\cap_{b>0}D(z,b)$. Since $m_n, m$ are analytic in $\bbC_+$ it is enough to prove that  $m_n$ converges to $m$ pointwise and that $m_n$ are uniformly bounded on compact subsets of $\bbC_+$. To prove the former, fix $\varepsilon>0$. Then by \eqref{1.15} we have $r(z,b)<\varepsilon$ for $b\geq b(\varepsilon)$ for some $b(\varepsilon)>0$. Next, there exists $N(\varepsilon)>0$ such that
		\begin{equation} \lb{1.18}
		s_k(n)=s_k, 0\leq k\leq \lfloor b(\varepsilon) \rfloor,
		\end{equation}
		for all $n\geq N(\varepsilon)$. Denoting the Weyl disk, Dirichlet and Neumann solutions corresponding to $s(n)$ by $D_n$, $\varphi_n, \theta_n$ respectively, we note that \eqref{1.18} yields
		\begin{equation}
		\varphi_n(z,x)=\phi(z,x), \theta(z,x)=\theta_n(z,x), x\in (0, b(\varepsilon)),\  n\geq N(\varepsilon).
		\end{equation}
		Hence, $D_n(z,b(\varepsilon))=D(z,b(\varepsilon))$ $n\geq N(\varepsilon)$, and since $m_n(z)\in D_n(z,b(\varepsilon))$, we get
		\begin{equation}\lb{1.19}
		m_n(z),m(z)\in D(z,b(\varepsilon)), n\geq N(\varepsilon).
		\end{equation}
		Therefore,
		\begin{equation}
		|m_n(z)-m(z)|\leq 2\varepsilon, n\geq N(\varepsilon).
		\end{equation}

		Next, utilizing  \eqref{1.19} with $\varepsilon=1$ we get
		\begin{equation}
		|m_n(z)-m(z)|\leq 2 r(z,b(1)), n\geq N(1).
		\end{equation}
		Since $r(z,b(1))$ and $m(z)$ are uniformly bounded on compacts, we infer that the sequence $\{m_n\}_{n\geq1}$ is uniformly bounded on compact subsets of $\bbC_+$.
	\end{proof}

\subsection{Restriction Maps}

Throughout this section $m_{\pm}(s,z)$, $\varphi_{\pm}(s,\cdot)$, $\theta_{\pm}(s,\cdot)$ denote the Weyl--Titchmarsh functions, Dirichlet, and Neumann solutions corresponding to  the sequence $s\in\cJ^{\bbZ}$. Let $\cR_{\pm}$ denote the restriction operators $\cR_{\pm}s:=\{s_n\}_{n\in\bbZ_{\pm}}$.

For a set $\cZ\subset \bbR$ we define
\begin{equation}
\cD(\cZ):=\{ s\in \cJ^{\bbZ}: m_-(s, E+i 0)=-\overline{m_+(s, E+i0)},\  \text{for Lebesgue a.e.}\,  E\in\cZ \}.
\end{equation}

\begin{theorem}\lb{thm3.6}
	Suppose that $\cZ$ has positive Lebesgue measure. Then
	\begin{itemize}
		\item [(i)] The mapping
		\begin{equation}
		\cR: \cR_{-}(\cD(\cZ))\rightarrow\cR_{+}(\cD(\cZ)), \ \cR: \cR_{-}(s)\mapsto \cR_{+}(s), s\in\cD(\cZ),
		\end{equation}
		is well-defined (i.e. single valued).
		\item[(ii)] $\cR$ is bijective and uniformly continuous.
		\item[(iii)] $\cD(\cZ)$ is $T$ invariant and closed with respect to the metric $d$, cf. \eqref{3.21}.
	\end{itemize}
\end{theorem}

\begin{proof}
The restrictions $\cR_{\pm}(s)$ uniquely determine the Weyl--Titchmarsh functions $m_{\pm} (s,\cdot)$ by a Borg--Marchenko type result Theorem \ref{thm1.4}. Recalling that every Herglotz function is determined by its boundary values on a set of positive Lebesgue measure, we get that the Weyl--Titchmarsh functions $m_{-} (s,\cdot)$ and $m_{+} (s,\cdot)$ are in one-to-one correspondence since
 \[m_-(s, E+i 0)=-\overline{m_+(s, E+i0)},\ a.e.\  E\in\cZ,\ |\cZ|>0.\]
\[
\begin{tikzcd}
& s \arrow{dl}{}\arrow{dr}{} \\
\cR_-(s)  \arrow{rr}{\cR}\arrow{d}{\text{Theorem \ref{thm1.4}}} && \arrow{ll} \cR_+(s)\arrow{d}{\text{Theorem \ref{thm1.4}}}\\
m_- (s,\cdot) \arrow{u}{} \arrow{rr}{|\cZ|>0} && \arrow{u}\arrow{ll} m_+(s, \cdot)
\end{tikzcd}
\]
This proves that $\cR$ is well-defined and bijective. In order to show that it is uniformly continuous it is enough to check that $\cD(\cZ)$ is closed (hence, compact). Indeed, in this case $\cR_{\pm}$ are bijective, continuous mappings between compact metric spaces, thus they are homeomorphic and $\cR$ is uniformly continuous, see the top part of the diagram. The closedness of $\cD(\cZ)$ follows from Lemma \ref{lemma3.5} and \cite[Lemma 5]{Ko85}.

The $T$-invariance of $\cD(\cZ)$ can be deduced from the following relation between $m_{\pm}(s,z)$ and  $m_{\pm}(Ts,z)$,
	\begin{align}
	m_{\pm}(Ts, z)&=\frac{\mu(Ts, 0^{\pm})\psi'_{\pm}(Ts, 0^{\pm})}{\psi_{\pm}(Ts, 0^{\pm})}=\frac{\mu(s, 1^{\pm})\psi'_{\pm}(s, 1^{\pm})}{\psi_{\pm}(s, 1^{\pm})}\\
	&=\frac{\cosh{\sqrt{-z}}\mu(s, 0^{\pm})\psi'_{\pm}(s, 0^{\pm})+ \frac{\sinh\sqrt{-z}}{\sqrt{-z}}\psi_{\pm}(s, 0^{\pm})}{{\sqrt{-z}\sinh\sqrt{-z}}\mu(s, 0^{\pm})\psi'_{\pm}(s, 0^{\pm})+ \cosh\sqrt{-z}\psi_{\pm}(s, 0^{\pm})}\\
	&=\frac{m_{\pm}(s,z)\cosh{\sqrt{-z}}+ \frac{\sinh\sqrt{-z}}{\sqrt{-z}}}{m_{\pm}(s,z){\sqrt{-z}\sinh\sqrt{-z}}+ \cosh\sqrt{-z} }.
	\end{align}
\end{proof}

\section{Ergodic Sturm--Liouville Operators in Impedance Form}\lb{section4}

This section concerns the spectral analysis of the full-line version of the operator \eqref{nn1.1n} with a dynamically defined sequence of sphere numbers $\{s_n\}_{n\in \bbZ}$. First, we show that the spectrum is given by the zero set of the Lyapunov exponent whenever the underlining cocycles are uniform, see  Lemma \ref{lemma4.3}. Further, assuming Boshernitzan's condition and using Kotani theory we derive the main result of this section stating that the spectrum is given by a generalized Cantor set of zero Lebesgue measure.

Let $\Omega$ be a compact metric space, let $T:\Omega\rightarrow\Omega$ be a homeomorphism, and suppose that $(\Omega, T)$\footnote{at this point $(\Omega, T)$ is not assumed to be a minimal subshift satisfying (B)} is uniquely ergodic with the unique $T-$invariant probability measure $\nu$. For a non-constant measurable function $f:\Omega\rightarrow\cJ$, let
		\begin{equation}
		\mu_{\omega}(x):=\sum_{n=-\infty}^{\infty}f(T^n\omega) f(T^{n+1}\omega)\chi_{[n, {n+1})}(x), x\in \bbR, \omega\in\Omega.
		\end{equation}
		Let $\tau_{\omega}$ be defined as in \eqref{n1.1n} with $\mu=\mu_{\omega}$ and introduce a self-adjoint  operator
		\begin{align}
		\begin{split}\lb{3.25}
			&H_{\omega}:\dom(H_{\omega})\subset L^2(\bbR; \mu_{\omega})\rightarrow L^2(\bbR; \mu_{\omega}), H_{\omega} u:= \tau_{\omega} u,\ u\in\dom(H_{\omega}),\\
		&\dom(H_{\omega})=\{u\in L^2(\bbR, \mu_{\omega}): u, \mu_{\omega} u'\in AC(\bbR), \tau_{\omega} u\in L^2(\bbR; \mu_{\omega})\}.
		\end{split}
		\end{align}
		
		For $\nu$ almost every $\omega\in\Omega$, the spectrum of $H_{\omega}$ is given be a deterministic set $\Sigma\subset\bbR$. Moreover, by \cite[Proposition V.2.4 and Remark V.2.5]{CL} there exists $\Sigma_{\bullet}\subset \bbR$ such that for $\nu-$a.e. $\omega$ one has
		\begin{equation}
		\sigma_{\bullet}(H_{\omega})=\Sigma_{\bullet},\ \bullet\in\{\text{ac, sc, pp, disc}\}.
		\end{equation}
		A much finer spectral analysis of the ergodic family $H_{\omega}, \omega\in\Omega$ is possible through the dynamical approach on which we focus next. First, notice that
		\begin{equation}\lb{n4.1}
		\tau_{\omega}u=Eu,\  u, \mu_{\omega}u\in AC_{loc}(\bbR_+)
		\end{equation}
		holds if and only for every $n\in\bbZ$, one has $-u''(x)=Eu(x), x\in(n-1, n)$ and
		\begin{align}
		&\begin{bmatrix}
		{u}(n^+)\\
		{u}'(n^+)
		\end{bmatrix}
		=
		\begin{bmatrix}
		\cos\sqrt{E}& \frac{\sin\sqrt{E}}{{\sqrt{E}}}\\
		{-{\frac{f(T^{n-1}\omega)\sqrt{E}\sin\sqrt{E}}{f(T^{n+1}\omega)}}}&{\frac{f(T^{n-1}\omega)\cos\sqrt{E}}{f(T^{n+1}\omega)}}
		\end{bmatrix}
		\begin{bmatrix}
		{u}((n-1)^+)\\
		{u}'((n-1)^+)
		\end{bmatrix}.
		\end{align}
		The latter can be rewritten as follows,
		\begin{align}
		\begin{bmatrix}
		f(T^{n+1}\omega){u}(n^+)\\
		{u}'( n^+)
		\end{bmatrix}&
		=
		\begin{bmatrix} \frac{f(T^{n+1}\omega)\cos\sqrt{E}}{f(T^{n-1}\omega)}
		& \frac{f(T^{n+1}\omega)\sin\sqrt{E}}{{\sqrt{E}}}\\
		{-{\frac{\sqrt{E}\sin\sqrt{E}}{f(T^{n+1}\omega)}}}&{\frac{f(T^{n-1}\omega)\cos\sqrt{E}}{f(T^{n+1}\omega)}}
		\end{bmatrix}
		\begin{bmatrix}
		f(T^{n-1}\omega){u}((n-1)^+)\\
		{u}'((n-1)^+)
		\end{bmatrix}\\
		&=\mathcal{M}^E(T^{n-1}\omega)\begin{bmatrix}
		f(T^{n-1}\omega){u}((n-1)^+)\\
		{u}'((n-1)^+)
		\end{bmatrix},
		\end{align}
		where the mapping $\cM^E:\Omega\rightarrow \text{SL}(2,\bbR)$ is given by
		\begin{equation}\lb{4.2aa}
				\mathcal{M}^E(\omega):=\begin{bmatrix} \frac{f(T^2\omega)\cos\sqrt{E}}{f(\omega)}
			& \frac{f(T^2\omega)\sin\sqrt{E}}{{\sqrt{E}}}\\
			{-{\frac{\sqrt{E}\sin\sqrt{E}}{f(T^2\omega)}}}&{\frac{f(\omega)\cos\sqrt{E}}{f(T^2\omega)}}
			\end{bmatrix}.
		\end{equation}
		The spectral properties of $H_{\omega}$ will be described using an SL$(2, \bbR)$ cocycle $(\omega, n)\mapsto \cM^E_n(\omega)$ over $T$ given by
		
		\begin{equation}\lb{4.2a}
	\mathcal{M}_n^E(\omega)=\begin{cases} 
{\mathcal{M}^E\left(T^{n-1}\omega \right)\times \cdots \times \mathcal{M}^E(\omega),} & {n \geq 1}, \\
{I_{2},} & {n=0}, \\
{(\mathcal{M}^E\left(T^{n}\omega \right))^{-1}\times \cdots \times \mathcal{M}^E(T^{-1}\omega)^{-1},} & {n \leq-1}.\end{cases}
		\end{equation}
	
		The Lyapunov exponent of this cocycle is defined by
		\begin{equation}\lb{n1.20}
		L(E):=\lim_{n\rightarrow\infty}\frac{1}{n}\int_{\Omega}\log\|\mathcal{M}_n^E(\omega)\|d\nu(\omega)\geq0,
		\end{equation}
		by Kingman's Subadditive Ergodic Theorem, one has
		\begin{equation}\lb{nn1.20}
		L(E)=\lim\limits_{n\rightarrow\infty} \frac1n\log\|M^E_n(\omega)\|,
		\end{equation}
		for $\nu$-almost every $\omega$. We say that the function $\cM^E$ is {\it uniform} if the limit \eqref{nn1.20} exists for all $\omega\in\Omega$ and the convergence is uniform.

		\subsection{Dynamical Description of The Spectrum}
	
		The main goal of this section is to describe the spectrum of $H_{\omega}$ in terms of the cocycle $(\omega, n)\mapsto M^E_n(\omega)$. To that end, we first define a set of energies corresponding to uniformly hyperbolic cocycles,
		\[\cU\cH=\left\{E \in \mathbb{R} \Big| \begin{matrix}
		\text{there exist\ } \lambda>1, C>0 \text { such that } \\
		\left\|\mathcal{M}_n^E(\omega)\right\| \geq C \lambda^{|n|}  \text{\ for all\ } \omega \in \Omega, n\in\mathbb{Z}
		\end{matrix} \right\}. \]
		We note that $\cU\cH$ is open by \cite[Corollary 2.2]{DFLY}. With the Lyapunov exponent $L(E)$ defined in \eqref{n1.20}, denote
		\begin{equation}
		\mathcal{Z}:=\left\{E\in\bbR: L(E)=0\right\},\ \mathcal{NUH}:=\left\{E \in \mathbb{R} : L(E)>0\right\} \backslash \mathcal{UH}.
		\end{equation}
		Then we arrive at the following disjoint partition of the real line
		\begin{equation}
		\mathbb{R}=\mathcal{Z}\sqcup \mathcal{NUH}\sqcup \mathcal{UH}.
		\end{equation}

		\begin{proposition}
			One has
			\begin{equation}\lb{n1.21}
			\bigcup_{\omega\in\Omega}\sigma(H_{\omega})=\bbR\setminus\cU\cH.
			\end{equation}
		\end{proposition}

		\begin{proof}
			Assume that $E_0\in\cU\cH$ and fix $\omega\in\Omega$. Since $\cU\cH$ is open there is an open interval $\cI\subset \cU\cH$ containing $E_0$.  For each $E\in\cI$ any solution of the equation $\tau_{\omega}u=Eu$ grows exponentially fast on at least one half-line. By \cite[Theorem 1.1]{BdMS}\footnote{our operator is of the form described in \cite[Section 3.2]{BdMS}, hence, the referenced result is applicable}, for spectrally almost every $\lambda\in\bbR$ there is an $L^2$-subexponentially bounded solution of $\tau_{\omega}f=\lambda f$. Since $\cI$ contains no such $\lambda$ we conclude that $E_0\not\in\sigma(H_{\omega})$
			That is, \[\bigcup_{\omega\in\Omega}\sigma(H_{\omega})\subset\bbR\setminus\cU\cH. \]
			
			Next, assume that $E\in\bbR\setminus\cU\cH$. Then the cocycle $(T,\cM^E)$ admits a bounded orbit, i.e., there exist $\omega\in\Omega$ and $v\in \{x\in \bbR^2: \|x\|=1 \}$ such that
			\begin{equation}
			\|\cM^E_n(\omega)v\|\leq 1,
			\end{equation}
			for all $n\in\bbZ$, see e.g. \cite[Theorem 1.2]{DFLY}. Therefore, the equation $\tau_{\omega}u=Eu$ admits a bounded solution $u_{0}$  satisfying all vertex conditions. Thus by Lemma \ref{lemma2.2} we have $E\in\sigma(H_{\omega})$.
		\end{proof}

	Assuming uniformity of the cocycle, we will show the spectrum of every $H_{\omega}$ is given by the right-hand side of \eqref{n1.21}, which in turn reduces to $\cZ$. To that end, we first prove a Sch'nol-type result in the deterministic setting.

	\begin{lemma}\lb{lemma2.2} Let $\tau,\mu$  be defined as in \eqref{n1.1n} for a two sided sequence $\{s_n\}_{n\in\bbZ}\subset \cJ$, and let $H$ be the corresponding self-adjoint operator in $L^2(\bbR)$ $($a two sided version of \eqref{nn1.1n}$)$. Suppose that $w$ is a subexponentially bounded solution
	\[ \tau w=Ew, E\in\bbR, w, \mu_{\omega}w\in AC_{loc}(\bbR), \]
	 that is, for every $\kappa>0$ there exists $C>0$ such that
		\begin{equation}\lb{new5}
		|w(x)|\leq C e^{\kappa x}\text{\ for all\ } x\in\bbR.
		\end{equation}
		Then $E\in\sigma(H)$.
	\end{lemma}

	\begin{proof}  Our strategy is to construct a Weyl sequence for $E$. First, we note that the quadratic form of $H$ (acting in $L^2(\bbR_+, \mu)$) is given by
		\begin{align}
		&\mathfrak h[u,v]:=\langle u',v'\rangle_{L^2(\bbR_+, \mu)},\dom(\mathfrak h)= H^1(\bbR).
		\end{align}
		Indeed, for every $u\in \dom(\mathfrak h)$ and every compactly supported $v\in \dom(H)$, $\supp(v)\in(-K,K)$, we have
		\begin{align}
		&\langle u,Hv\rangle_{L^2(\bbR_+,\mu)}=-\langle u,(\mu v')'\rangle_{L^2(\bbR_+,dx)}\\
		&=-\sum_{j=-K+1}^{K}\int_{{j-1}}^{{j}} s_j s_{j-1} \overline{u(x)}v''(x) \, dx=\sum_{j=-K+1}^{K}\int_{{j-1}}^{{j}} s_j s_{j-1} \overline{u'(x)}v'(x)dx\\
		&\quad+\sum_{j=-K+1}^{K}s_js_{j+1}\overline{u(j^+)} v'(j^+)-s_js_{j-1}\overline{u(j^-)} v'(j^-)\no\\
		&=\langle u',v'\rangle_{L^2(\bbR, \mu)}+\sum_{j=-K+1}^{K}(s_js_{j+1}v'(j^+)-s_js_{j-1}v'(j^-))\overline{u(j^+)} =\langle u',v'\rangle_{L^2(\bbR, \mu)}.
		\end{align}
		For every $n\in\bbN$, let $\varphi_n\in C^{\infty}(\bbR)$ be a mollifier taking values in $[0,1]$ and satisfying
		\begin{equation}
		\varphi_n(x):=\begin{cases}
		1& 0\leq |x|\leq n, \\
		0&|x|\geq n+1;
		\end{cases}
		\text{\ \ and\ \ }\sup_{n\geq 1}\|\varphi'_n\|_{L^{\infty}(\bbR, dx)}<\infty.
		\end{equation}
		We claim that \[\left\{\frac{w\varphi_n}{\|w\chi_{(-n,n)}\|}_{L^2(\R, \mu)}\right\}_{n\geq 1}\subset \dom(\mathfrak h)\] is a Weyl sequence for $E$. First, we note that the $L^2(\bbR, \mu)$ norm of all elements of this sequence is at least one. Next, fix $u\in \dom(\mathfrak h)$, $\|u\|_{H^1(\bbR)}\leq 1$, and denote $I_n:=(n, n+1)\cup (-n-1,-n)$. Then one has
		\begin{align}
	\begin{split}\lb{new4}
	&{\mathfrak h[u,\varphi_n w]-E\langle u,\varphi_n w\rangle_{L^2(\R, \mu)}}\\
	&={\langle u',\varphi_n w'\rangle_{L^2(\R, \mu)}+\langle u,\varphi_n' w'\rangle_{L^2(\R, \mu)}-E\langle u,\varphi_n w\rangle_{L^2(\R, \mu)}}\\
	&= \int_{-n}^{n} \overline{u'(x)}w'(x)\mu(x)dx-E\int_{-n}^{n} \overline{u(x)}w(x)\mu(x)dx+\langle u',\chi_{I_n}\varphi_n w'\rangle_{L^2(\R, \mu)}\\
	&\quad-E\langle u,\chi_{I_n}\varphi_n w\rangle_{L^2(\R, \mu)}+\langle u,\varphi'_n\chi_{I_n} w'\rangle_{L^2(\R, \mu)}\\
	&= \langle u, (\tau w-Ew)\chi_{(-n, n)}\rangle_{L^2(\R, \mu)}\\
	&+s_ns_{n-1}\overline{u(n^-)}w'(n^-)-s_{-n}s_{-n+1}\overline{u(-n^+)}w'(-n^+)+\langle u',\chi_{I_n}\varphi_n w'\rangle_{L^2(\R, \mu)}\\
	&\quad-E\langle u\chi_{I_n},\varphi_n w\rangle_{L^2(\R, \mu)}+\langle u, \chi_{I_n}\varphi'_nw'\rangle_{L^2(\R, \mu)}.
	\end{split}
	\end{align}
		Sobolev--type inequalities, see, for example, \cite[IV.1.2]{K80}, yield
		\begin{align}
		\begin{split}
		&|u(x_0)|\lesssim \|u\|_{H^1(I_n)}, x_0\in\{-n,n\},\\
		&|w'(x_0)|\lesssim \|w\|_{H^1(I_n)}\lesssim \|w\|_{L^2(I_n, dx)}, \lb{new2}
		\end{split}
		\end{align}
		where we used $-w''(x)=Ew(x)$, $x\in I_n$ to get the last inequality in \eqref{new2}.
		Combining \eqref{new4}, \eqref{new2}, and the Cauchy--Schwarz inequality, we get
		\begin{align}
		&\frac{|{\mathfrak h[u,\varphi_n w]-E\langle u,\varphi_n w\rangle_{L^2(\R, \mu)}}|}{\|w\chi_{(-n,n)}\|_{L^2(\R, \mu)}}\lesssim\frac{\|u\|_{H^1(I_n)} \|w\|_{H^1(I_n)}}{\|w\chi_{(-n,n)}\|_{L^2(\R, dx)}}\\
		&\qquad\lesssim \frac{\|w\chi_{(n, n+1)\cup (-n-1,-n)}\|_{L^2(\R, dx)}}{\|w\chi_{(-n,n)}\|_{L^2(\R, dx)}}.\lb{new1}
		\end{align}
		The right-hand side of \eqref{new1} tends to zero as $n\rightarrow\infty$, for otherwise the norm \[\|w\chi_{(n, n+1)\cup (-n-1,-n)}\|_{L^2(\R, dx)}\] would grow exponentially, contradicting \eqref{new5}.
	\end{proof}
	
		\begin{lemma}\lb{lemma4.3}
		Suppose that $\cM^E$ is uniform for all $E\in\bbR$. Then
		\begin{equation}\lb{n1.22}
		\sigma(H_{\omega})= \cZ \text{\ for all\ }\omega\in\Omega.
		\end{equation}
	\end{lemma}

	\begin{proof}
		By assumption we have $\cN\cU\cH=\emptyset$. Therefore, due to \eqref{n1.21} it is enough to show that $\cZ\subset \sigma(H_{\omega})$ for all $\omega\in\Omega$. The latter follows form the fact that every solution of $\tau_{\omega}u=Eu$ is subexponentially bounded for every $E\in\cZ$, $\omega\in\Omega$ (since $\cM^E$ is uniform) and Lemma \ref{lemma2.2}.
	\end{proof}
	
	Next, we switch to Kotani's description of $\Sigma_{ac}$. Recall that the essential closure of a set $S\subset \bbR$ is given by
	\[\overline{S}^{ess}:=\{E\in\bbR: |S\cap(E-\varepsilon, E+\varepsilon)|>0\text{ for every}\ \varepsilon>0\}.\]

	\begin{theorem}\lb{theorem4.4}\emph{(Ishii--Pastur--Kotani)}
	The almost sure absolutely continuous spectrum of $H_{\omega}$ is given by the essential closure of the zero set of the Lyapunov exponent, that is, $\Sigma_{ac}=\overline{\cZ}^{ess}$. Moreover, for $\nu$-almost every $\omega$, one has
	\begin{equation}\lb{4.9}
	m_-(\omega, E+i 0)=-\overline{m_+(\omega, E+i0)},\ \emph{Leb-a.e.}\  E\in\cZ.
	\end{equation}
	\end{theorem}

	Kotani's original proof is formally recorded for Schr\"odinger operators with ergodic $L^1_{loc}(\bbR)$ potentials, see \cite{Ko}. However, it extends directly to the ergodic family $H_{\omega}$ defined in \eqref{3.25}. Let us indicate the only two nontrivial adjustments to be made in Kotani's text. First, since we are dealing with Sturm--Liouville operators in impedance form, the derivatives $\psi'_{\lambda}, \varphi'_{\lambda}$ of the Dirichlet and Neumann solutions (here we use Kotani's notation, see \cite[$\S$1]{Ko}) should be replaced by the quasi-derivatives $\mu_{\omega}\psi'_{\lambda}, \mu_{\omega}\varphi'_{\lambda}$ which are locally absolutely continuous. Second, the uniform boundedness of the Weyl--Titchmarsh functions $m_{\pm}$ discussed in \cite[Lemma 1.2]{Ko} follows (exactly as in \cite[Lemma 6.2]{LSS}) from the continuity of $m-$functions with respect to $\omega$ which has been established in Lemma \ref{lemma3.5}.
	
	\subsection{Cantor Spectrum For the Full Line Model} In this section we specialize the dynamical system $(\Omega, T)$ to the class of minimal subshifts over $\cJ$ satisfying the Boshernitzan condition (B), and $f(\omega):=\omega_0$.

	\begin{theorem}\lb{theorem4.5} Let $(\Omega,T)$ be an aperiodic, minimal subshift over $\cJ$ satisfying Boshernitzan's condition \eqref{BC}. Then there exists a generalized Cantor set $\Sigma\subset\bbR$ of Lebesgue measure zero such that
	\begin{equation}
	\sigma(H_{\omega})=\Sigma\text{\ for all\ }\omega\in\Omega.
	\end{equation}
	\end{theorem}

	\begin{proof}
    Since $(\Omega,T)$ satisfies (B) and since $\cM^E$ is locally constant we get that it is uniform for every $E\in\bbR$ by \cite[Theorem 1]{DL1}. Therefore, Lemma~\ref{lemma4.3} is applicable in the present setting and one has $\sigma(H_{\omega})=\cZ$ for all $\omega\in\Omega$.

    We claim that $\cZ$ has zero Lebesgue measure. Seeking contradiction let us assume that $|\cZ|>0$. By Theorem~\ref{theorem4.4} the assertion \eqref{4.9} holds for $\omega\in\wti\Omega\subset\Omega$, $\nu(\wti\Omega)=1$. In other words, $\wti\Omega\subset\cD(\cZ)$. Then combining Theorem~\ref{thm3.6} and $|\cZ|>0$, we infer that the mapping
    \begin{equation}
    \cR: \{f(T^n\omega) f(T^{n+1}\omega)\}_{n<0}\mapsto \{f(T^n\omega) f(T^{n+1}\omega)\}_{n\geq0},\ \omega\in\wti\Omega,
    \end{equation}
    is well-defined and continuous with respect to the discrete metric $d$, cf.~\eqref{3.21}.  Pick an $\omega\in\wti\Omega$ and denote $s_n:=f(T^n\omega)$ for $n\in\bbZ$. By uniform continuity of $\cR$ there exists $k=k(\cJ)\in\bbN$ (which is $\omega$-independent) such that the collection  $\{s_{-k}s_{-k+1}; ...; s_{-1}s_{0}\}$ uniquely determines $s_{0}s_{1}$. Shifting and repeating this argument, we deduce that $\{s_{-k}s_{-k+1}; ...; s_{-1}s_{0}\}$  determines the entire sequence $\{s_{n}s_{n+1}\}_{n\in\bbZ}$. Thus
    \begin{equation}
    \#\{\{f(T^n\omega) f(T^{n+1}\omega)\}_{n\in\bbZ}: \omega\in\Omega \}\leq |\cJ|^{k+1}<\infty.
    \end{equation}
    In particular, the sequence $\{s_{n}s_{n+1}\}_{n\in\bbZ}$ is periodic with some period $p$. Then since $\cJ$ is finite, the sequence $\{s_{n}\}_{n\in\bbZ}$ is periodic with period $q\leq 2|\cJ|^2p$ (in general $p<q$, e.g.: $...141222\  141222\  ...$ here $p=3$ and $q=6$). This contradicts the fact the $\Omega$ is an aperiodic subshift. Therefore, $|\Sigma|=|\cZ|=0$, in particular the closed set $\Sigma$ does not contain intervals, hence, it is nowhere dense. Moreover, $\nu$ almost surely, the discrete spectrum spectrum $\sigma_{\emph{disc}}(H_{\omega})$ is empty, cf.~\cite[Proposition V.2.8]{CL}, so $\Sigma$ does not have isolated points, hence, it is a generalized Cantor set as asserted.
	\end{proof}
\section{Spectral Decomposition of the Two-sided Fibonacci Graph}
In this section we focus on the two-sided version of the graph $\Gamma$ corresponding to an element $\omega$ of the Fibonacci subshift $\Omega$ introduced in Example \ref{ex2.3}, see Figure \ref{Fig2}. 
\begin{figure}[h]
	\centering
	\begin{tikzpicture}
	\node[fill=black,circle,scale=0.3](1) at (1,0){};
	\node[fill=black,circle,scale=0.3](2) at (2,0.7){};
	\node[fill=black,circle,scale=0.3](3) at (2,-0.7){};
	\node[fill=black,circle,scale=0.3](4) at (3,0){};
	\node[fill=black,circle,scale=0.3](5) at (4,0){};
	\node[fill=black,circle,scale=0.3](6) at (5,0.7){};
	\node[fill=black,circle,scale=0.3](7) at (5,-0.7){};
	\node[fill=black,circle,scale=0.3](8) at (6,0){};
	\node[fill=black,circle,scale=0.3](9) at (6,0){};
	\node[fill=black,circle,scale=0.3](10) at (7,0.7){};
	\node[fill=black,circle,scale=0.3](11) at (7,-0.7){};
	\node[fill=black,circle,scale=0.3](12) at (8,0){};
	\node[fill=black,circle,scale=0.3](13) at (9,0){};
	\draw[-,thin] (1) to (2);
	\draw[-,thin] (1) to (3);
	\draw[-,thin] (2) to (4);
	\draw[-,thin] (3) to (4);
	\draw[-,thin] (4) to (5);
	\draw[-,thin] (5) to (6);
	\draw[-,thin] (5) to (7);
	\draw[-,thin] (6) to (8);
	\draw[-,thin] (7) to (8);
	\draw[-,thin] (9) to (10);
	\draw[-,thin] (9) to (11);
	\draw[-,thin] (10) to (12);
	\draw[-,thin] (11) to (12);
	\draw[-,thin] (12) to (13);
	\filldraw (0,0) circle (.7pt);
	\filldraw (.2,0) circle (.7pt);
	\filldraw (.4,0) circle (.7pt);
	\filldraw (9.8,0) circle (.7pt);
	\filldraw (9.4,0) circle (.7pt);
	\filldraw (9.6,0) circle (.7pt);
	\end{tikzpicture}
	\caption{Two-sided Fibonacci graph,\ $s=\{...12112121...\}$}\lb{Fig2}
\end{figure}
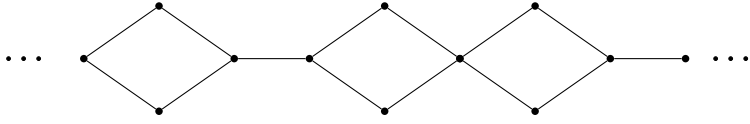

First, we discuss the spectral properties of Fibonacci hamiltonians. Let us recall relevant notions and fix notation. Let $s\in\{1,2\}^{\bbZ_+}$ be the one-sided sequence invariant under the Fibonacci substitution $\cS$, cf. Example \ref{ex2.3}, and let $\omega(s)\in\Omega$ be a two-sided sequence which matches $s$ on $\bbZ_+$. Let  $F_n$ be  the $n$-th element of the Fibonacci sequence with $F_0:=1$, $F_1:=2$ (note that $F_n=|\cS^n(1)|$ ). Setting $f(\omega):=\omega_0$ in \eqref{4.2aa} we denote $M(n,E):=\cM^E_{F_n}(\omega(s))$, $n\geq0$, cf. \eqref{4.2a}. The key to the spectral analysis of Fibonacci hamiltonians is the following recurrence relation
\begin{equation}\lb{7.5}
M(n+1,E)=M(n-1, E)M(n,E), n\geq 3. 
\end{equation}
We stress that the transfer matrix \eqref{4.2aa} depends on the window (\fbox{${\omega_0, \omega_1, \omega_2}$}...)  of size {\it three}. Consequently, $M(n, E)$ depends on the first $F_n+2$ elements of $s$. Thus, the recurrent relation
\begin{equation}\lb{100}
\cS^{n+1}(1)=\cS^{n}(1)\cS^{n-1}(1),
\end{equation}
does not automatically yield \eqref{7.5} for {\it all} $n\geq 1$. However, it does imply that the extra two elements required in the definition of $M(n, E)$ are $n$-independent and, in fact, given by\[s_{F_n+1}=1,\ s_{F_n+2}=2,\ n\geq 2.\]  
Now, this observation together with \eqref{100} do yield \eqref{7.5}. 
\begin{theorem}\lb{thm5.1}
Let $\Omega$ be the Fibonacci subshift. Then $\sigma_{pp}(H_{\omega})=\emptyset$ for all $\omega\in\Omega$.  In fact, the spectrum of $H_{\omega}$ is purely singular continuous. 
\end{theorem}
\begin{proof}Define \[x_n(E):=\frac{1}{2}\text{tr}(M(n,E)).\]
	Then \eqref{7.5} together with the Cayley--Hamilton theorem yield
	\begin{equation}\lb{7.6}
	x_{n+1}(E)=2x_{n}(E)x_{n-1}(E)-x_{n-2}(E), n\geq 4.
	\end{equation}
Let 
\[B:=\big\{E\in\bbR: \sup_{n\in\bbZ_+}|x_n(E)|<+\infty \big\}. \]It is enough to prove that $\Sigma\subset B$ (see Theorem \ref{theorem4.5} for definition of $\Sigma$) and that $B$ contains no eigenvalues of $H_{\omega(s)}$. Let $H_n$ denote the periodic hamiltonian corresponding to the sequence of sphere numbers obtained by two-sided periodic extension of the patch $\cS^n(1)$. We claim that 
\begin{equation}\lb{7.1}
\Sigma\subset \bigcap_{n\geq 0}\overline{\bigcup_{k\geq n} \sigma(H_k) }\subset B.
\end{equation} 
Let us prove the first inclusion in \eqref{7.1}. Fix $n\geq 0$ and suppose that \[E\not\in \overline{\bigcup_{k\geq n} \sigma(H_k) },\] then
\begin{equation}
\varkappa:=\sup_{k\geq n}\|(H_k-E)^{{-1}}\|_{\cB(L^2(\bbR))}<+\infty.
\end{equation}
Thus, for every compactly supported $\psi\in\dom(H_{\omega(s)})$ we have $\psi \in\dom(H_k)$ and
\[\|H_{k}\psi-E\psi \|_{L^2(\bbR)}\geq\varkappa^{-1}\|\psi\|_{L^2(\bbR)}, \]
for sufficiently large $k$. For such $k$ we have $\|H_{k}\psi-E\psi \|_{L^2(\bbR)}=\|H_{\omega(s)}\psi-E\psi \|_{L^2(\bbR)}$, and infer
\begin{equation}\lb{5.1new}
\|H_{\omega(s)}\psi-E\psi \|_{L^2(\bbR)}\geq\varkappa^{-1}\|\psi\|_{L^2(\bbR)}.
\end{equation}
Since the set of compactly supported functions from $\dom(H_{\omega(s)})$ is a core of $H_{\omega(s)}$, the inequality \eqref{5.1new} holds for all $\psi\in\dom(H_{\omega(s)})$. Furthermore, since $H_{\omega(s)}$ is self-adjoint we get $E\not\in \sigma(H_{\omega(s)})=\Sigma$ which proves the first inclusion in \eqref{7.1}. 

 As was noted in the proof of \cite[Proposition 6.3]{DFG}, the assertion \cite[Proposition 12.8.6]{S05} together with \eqref{7.6} yield the following criterion: $E\not\in B$ if and only if there exists $n$ such that 
 \begin{equation}\lb{7.2}
|x_{n+1}(E)|>1, |x_{n}(E)|>1, |x_{n+1}(E)x_{n}(E)|>|x_{n-1}(E)|,
 \end{equation}
  and in this case 
  \begin{equation}\lb{7.3}
|x_k(E)|>1, k\geq n.
  \end{equation}
   To prove the second inclusion in \eqref{7.1} let us fix $E\not\in B$. Since $x_j(E)$ is continuous with respect to $E$, the inequalities in \eqref{7.2}  hold in a small neighborhood of $E$. Thus \eqref{7.3} also holds in the same neighborhood (uniformly for $k\geq n$!). Then by Floquet--Bloch theory $E\not\in \overline{\bigcup_{k\geq n} \sigma(H_k)}$. 
   
   The fact that $\tau_{\omega(s)}u=Eu$ has no square integrable solution for $E\in B$ follows form the standard argument invoking Gordon's lemma and the boundedness  of $x_n(E)$, see, for example, \cite[Proposition 2]{Su}. 
\end{proof}

Our next objective is to obtain the complete spectral decomposition of the Kirchhoff Laplacian corresponding to the two-sided Fibonacci Graph $\Gamma$. To that end, we first note that the decomposition result of Kostenko and Nicolussi \cite[Theorem 3.5]{KN} can be extended to the two-sided setting. Let $\{s_n\}_{n\in\bbZ}\subset\bbN$ be an arbitrary sequence of natural numbers. For each $n\in\bbZ$ let $S_n$ denote a set of  $s_n$ distinct vertices. Let $\Gamma$ be the graph with the set of vertices $\cV:=\cup_{n\in\bbZ} S_n$ and such that two vertices $u,v\in\cV$ are adjacent if and only if
\[(u,v)\in\bigcup_{n\in \bbZ}S_n\times (S_{n-1}\cup S_{n+1}),\] see, e.g., Figure  \ref{Fig2}. 
\begin{proposition}\lb{prop5.2}   The Kirchhoff Laplacian  $\bbH$ on the graph $\Gamma$ introduced above is unitarily equivalent to 
	\begin{equation}
	H\bigoplus_{n\in\bbZ}\oplus _{j=1}^{s_{n}-1}  H_n^1 \oplus\bigoplus_{n\in\bbZ}\oplus_{j=1}^{(s_{n}-1)(s_{n+1}-1)}H_n^2,
	\end{equation}
where  $H_n^1, H_n^2$ are as in \eqref{9.1}, \eqref{9.2}, respectively, and  $H$ is the self-adjoint Sturm--Liouville operator acting in $L^2(\bbR, \mu)$ and corresponding to 
	\begin{align}
\begin{split}
&\tau = -\frac{1}{\mu(x)}\left(\frac{d}{dx}\mu(x)\frac{d}{dx}\right),\\
&\mu(x):=\sum_{n=-\infty}^{\infty}{s}_n {s}_{n+1}\chi_{[n, {n+1})}(x), x\in \bbR.
\end{split}
\end{align}
\end{proposition}

The proof of this assertion is similar to that of \cite[Theorem 3.5]{KN}. Let us highlight the main ingredient of the proof by example. Let $\bbH$ be the Kirchhoff Laplacian on the two-side  graph $\Gamma$ introduced above and corresponding to an element $\omega$ of the Fibonacci subshift $\Omega\subset\{1,2\}^{\bbZ}$. Such a graph consists of tiles and line segments as shown on Figure \ref{Fig2}.  One has 
\begin{equation}\lb{10}
L^2(\Gamma)=\cF_{sym}\oplus (\cF_{sym})^{\perp},
\end{equation}
where $\cF_{sym}$ consists of $L^2(\Gamma)$ functions that are horizontally symmetric on each tile (for precise definition see \cite[(2.11)]{KN} with ``$n\geq 0$" replaced by ``$n\in\bbZ$", and $s_n\in\{1,2\}$). Then the operator block of $\bbH$ corresponding to $\cF_{sym}$ is unitarily equivalent to the full line Strum--Liouville operator $H$. The operator block of $\bbH$ corresponding to $(\cF_{sym})^{\perp}$ is unitarily equivalent to 
\begin{equation}\lb{5.1a}
\bigoplus_{n\in\bbZ}\oplus _{j=1}^{s_{n}-1}  H_n^1 \oplus\bigoplus_{n\in\bbZ}\oplus_{j=1}^{(s_{n}-1)(s_{n+1}-1)}H_n^2.
\end{equation}
\begin{proposition}Let $\bbH$ be as above. Then the singular continuous subspace of $\bbH$ is $\cF_{sym}$, the pure point subspace is $\cF_{sym}^{\perp}$, and the absolutely continuous part is trivial. 
\end{proposition}
\begin{proof}
The spectra of $H_n^1, H_n^2$ are discrete as before for every $n\in\bbZ$. By Theorem \ref{thm5.1} the spectrum of $H$ is purely singular continuous. Then the orthogonal decomposition \eqref{10} yields the assertion. 
\end{proof}

	

\begin{thebibliography}{EGNT}
		\bibitem[AW11]{AW11} M.\ Aizenman,\ S.\ Warzel,   {\it Absence of mobility edge for the Anderson random potential on tree graphs at weak disorder,} EPL (Europhysics Letters) {\bf 96} (2011), 37004.
		%
		\bibitem[AW13]{AW13} M.\ Aizenman,\ S.\ Warzel,   {\it Resonant delocalization for random Schr\"odinger operators on tree graphs,} J. Eur. Math. Soc. (JEMS) {\bf 15} (2013), 1167--1222.
		%
		\bibitem[Ben]{Ben} C.\ Bennewitz, {\it A proof of the local Borg--Marchenko Theorem}, Commun. Math. Phys. {\bf 218} (2001), 131--132.
		%
		\bibitem[Ber]{Ber} G.\ Berkolaiko, {\it An elementary introduction to quantum graphs}, in "Geometric and Computational Spectral Theory", Contemporary Mathematics, 700, AMS 2017.
		%
		\bibitem[BK]{BK} G.\ Berkolaiko, P.\ Kuchment,  {\it Introduction to Quantum Graphs}, Mathematical Surveys and Monographs, vol. 186, AMS, Providence, 2012.
		%
		\bibitem[Bo]{B} G.\ Borg, {\it Uniqueness theorems in the spectral theory of $-y''+(\lambda-q(x))y=0$}, In: Proc. 11th Scandinavian Congress of Mathematicians. Oslo: Johan Grundt Tanums Forlag, 1952, 276--287.
		%
		\bibitem[BMLS]{BMLS} A.\ Boutet de Monvel, D.\ Lenz, P.\ Stollmann, {\it Sch'nol's theorem for strongly local forms},  Isr. J. Math. {\bf 173} (2009), 189--211.
		%
		\bibitem[BdMS]{BdMS} A.\ Boutet de Monvel, P.\ Stollmann, {\it Eigenfunction expansions for generators of Dirichlet forms},
		J. Reine Angew. Math. {\bf 561} (2003), 131--144.
		%
		\bibitem[Br]{Br07} J.\ Breuer, {\it Singular continuous spectrum for the Laplacian on certain sparse trees}, Commun. Math. Phys. {\bf 219} (2007), 851--857.
		%
		\bibitem[BF]{BF} J.\ Breuer, R.\ Frank, {\it Singular spectrum for radial trees}, Rev. Math. Phys. {\bf 21} (2009), 929--945.
		%
		\bibitem[BKe]{BKe} J.\ Breuer, M.\ Keller,  {\it Spectral analysis of certain spherically homogeneous graphs}, Operators and Matrices {\bf7} (2013), 825--847.
		%
		\bibitem[BL]{BL} J.\ Breuer, N.\ Levi, {\it On the decomposition of the Laplacian on metric graphs}, arXiv:1901.00349 (2019).
		%
		\bibitem[CL]{CL} R.\ Carmona,\ J.\ Lacroix,  {\it Spectral Theory of Random Schr\"odinger Operators}, Birkh\"auser, Boston, 1990.
		%
		\bibitem[D07]{D1} D.\ Damanik, {\it Lyapunov exponents and spectral analysis of ergodic Schr\"odinger operators: A survey of Kotani theory and its applications}, Spectral theory and mathematical physics: a Festschrift in honor of Barry Simon's 60th birthday, 539--563, Proc. Sympos. Pure Math., 76, Part 2, Amer. Math. Soc., Providence, RI, 2007.
		%
		%
		\bibitem[D17]{D2} D.\ Damanik, {\it Schr\"odinger operators with dynamically defined potentials}, Ergodic Theory Dynam. Systems {\bf 37} (2017), 1681--1764.
        %
        \bibitem[DF]{DF} D.\ Damanik, J.\ Fillman, \textit{Spectral Theory of Discrete One-Dimensional Ergodic Schr\"odinger Operators}, monograph in preparation.
        %
         \bibitem[DFG]{DFG} D.\ Damanik, J.\ Fillman, A.\ Gorodetski, \textit{Continuum Schr\"odinger operators associated with aperiodic subshifts}, Ann. Henri Poincare {\bf 15} (2014), 1123--1144. 
        %
        \bibitem[DFLY]{DFLY} D.\ Damanik, J.\ Fillman, M.\ Lukic, W.\ Yessen, \textit{Characterizations of uniform hyperbolicity and spectra of CMV matrices}, Discrete Contin. Dyn. Syst. Ser. S 9 {\bf4} (2016), 1009--1023.
        %
		\bibitem[DFS]{DFS} D.\ Damanik, J.\ Fillman, S.\ Sukhtaiev, {\it Localization for Anderson models on metric and discrete tree graphs},  	 arXiv:1902.07290 (2019).
		%
		\bibitem[DL06a]{DL1} D.\ Damanik, D.\ Lenz, {\it A condition of Boshernitzan and uniform convergence in the multiplicative ergodic theorem}, Duke Math. J. {\bf 133} (2006), 95--123.
		%
		\bibitem[DL06b]{DL2} D.\ Damanik, D.\ Lenz, {\it Zero-measure Cantor spectrum for Schr\"odinger operators with low-complexity potentials}, J. Math. Pures Appl. {\bf 85} (2006), 671--686.
		%
		\bibitem[DS]{DS} D.\ Damanik, S.\ Sukhtaiev, {\it Anderson localization for radial tree graphs with random branching numbers},  arXiv:1803.06037; to appear in J. Funct. Anal. https://doi.org/10.1016/j.jfa.2018.11.007.
		%
		\bibitem[EGNT]{EGNT1} J.\ Eckhardt, F.\ Gesztesy, R.\ Nichols, G.\ Teschl, {\it Inverse spectral theory for Sturm--Liouville operators with distributional potentials}, J. London Math. Soc. {\bf 88} (2013), 801--828.
		%
		\bibitem[ESS]{ESS} P.\ Exner, C.\ Seifert, P.\ Stollmann, {\it Absence of absolutely continuous spectrum for the Kirchhoff Laplacian on radial trees}, Ann. Henri Poincar\'e {\bf 15} (2014), 1109--1121.
		%
		\bibitem[GL]{GL} I.\ M.\ Gel'fand, B.\ M.\ Levitan, {\it On the determination of a differential equation from its spectral function}, Izv. Akad. Nauk SSR. Ser. Mat. {\bf 15} (1951), 309--360 (Russian); English transl. in Amer. Math. Soc. Transl. Ser. 2 {\bf 1} (1955), 253--304.
		%
		\bibitem[GS]{GS} F.\ Gesztesy, B.\ Simon, {\it On local Borg--Marchenko uniqueness results}, Commun. Math. Phys. {\bf 211} (2000), 273--287.
		%
		\bibitem[GLN]{GLN} R.\ Grigorchuk, D.\ Lenz, T.\ Nagnibeda,  {\it Spectra of Schreier graphs of Grigorchuk's group and Schroedinger operators with aperiodic order}, Math. Ann. {\bf 370} (2018), 1607--1637. 
		%
		\bibitem[Ka]{K80} T.\ Kato, {\it Perturbation Theory for Linear Operators}, Springer, Berlin, 1980.
		%
		\bibitem[Kl]{K3} A.\ Klein, {\it Extended states in the Anderson model on the Bethe lattice}, Adv. Math. {\bf 133} (1998), 163--184.
		%
		\bibitem[KN]{KN} A.\ Kostenko, N.\ Nicolussi, {\it Quantum graphs on radially symmetric antitrees}, arXiv:1901.05404 (2019).
		%
		\bibitem[Ko85a]{Ko85} S.\ Kotani, {\it Support theorems for random Schr\"odinger operators}, Commun. Math. Phys. {\bf 97} (1985), 443--452.
		%
		\bibitem[Ko85b]{Ko} S.\ Kotani, {\it One-dimensional random Schr\"odinger operators and Herglotz functions}, Probabilistic methods in mathematical physics (Katata/Kyoto, 1985), pp. 219--250, Academic Press, Boston, 1987.
		%
		\bibitem[LSS]{LSS} D.\ Lenz, C.\ Seifert, P.\ Stollmann, {\it Zero measure Cantor spectra for continuum one-dimensional quasicrystals}, J. Diff. Eq. {\bf 256} (2014), 1905--1926.
		%
		\bibitem[M]{M} V.\ A.\ Marchenko, {\it Some questions in the theory of one-dimensional linear differential operators of the second order}, I.  Trudy Moskv. Mat. Obsch. {\bf 1} (1952), 327--420 (in Russian); English transl. in Am. Math.
		Soc. Transl. (2) {\bf 101} (1973), 1--104.
		%
		\bibitem[S96]{S96} B.\ Simon, {\it Operators with singular continuous spectrum, VI. Graph Laplacians and Laplace--Beltrami operators}, Proc. Amer. Math. Soc. {\bf 4} (1996), 1177--1182.
		%
		\bibitem[S99]{S} B.\ Simon, {\it A new approach to inverse spectral theory, I. Fundamental formalism}, Ann. Math. {\bf 150} (1999), 1029--1057.
		%
		\bibitem[S05]{S05} B.\ Simon, {\it Orthogonal polynomials on the unit circle. Part 2. Spectral theory.} Colloquium Publications, vol. 54, Part 2. American Mathematical Society, Providence, RI, (2005). 
		%
		\bibitem[Su]{Su} A.\ S\"ut{\H o}, {\it  The spectrum of a quasiperiodic Schrödinger operator}, Comm. Math. Phys. {\bf111} (1987), 409--415.
		%
		\bibitem[T09]{T09} G.\ Teschl, {\em Mathematical Methods in Quantum Mechanics; With Applications to Schr\"odinger Operators}, 2nd ed., Amer. Math. Soc., Providence, 2014.
	\end{thebibliography}
\end{document}